\documentclass[11pt]{article}
\usepackage[T1]{fontenc}
\usepackage[latin1]{inputenc}
\usepackage{lmodern}
\usepackage{amsmath, latexsym, amsfonts, amssymb, amsthm, amscd, mathrsfs}
\usepackage{dsfont}

\usepackage{parskip}
\begingroup
    \makeatletter
    \@for\theoremstyle:=definition,remark,plain\do{%
        \expandafter\g@addto@macro\csname th@\theoremstyle\endcsname{%
            \addtolength\thm@preskip\parskip
            }%
        }
\endgroup

\usepackage{xcolor}
\usepackage[shortlabels]{enumitem}
\usepackage{hyperref}
\usepackage[all]{hypcap}
\hypersetup{
  colorlinks=true,       			
  linkcolor=blue!50!black,    
  citecolor=red!50!black,     
  filecolor=magenta,     			
  urlcolor=green!50!black     
}

\usepackage[numbers,sort&compress]{natbib}
\usepackage{doi}

\newtheorem{theorem}{Theorem}

\newtheorem{proposition}[theorem]{Proposition}
\newtheorem{lemma}[theorem]{Lemma}
\newtheorem{remark}[theorem]{Remark}
\newtheorem{definition}[theorem]{Definition}
\newtheorem{example}[theorem]{Example}

\newcommand{\E}{\mathbb{E}}
\newcommand{\N}{\mathbb{N}}
\newcommand{\R}{\mathbb{R}}
\newcommand{\cN}{\mathcal{N}}
\newcommand{\cP}{\mathcal{P}}
\newcommand{\ed}{\stackrel{(d)}{=}}
\renewcommand{\P}{\mathbb{P}}
\newcommand{\1}{\mathbb{1}}
\newcommand{\abs}[1]{\lvert #1 \rvert}
\renewcommand{\rho}{\varrho}
\renewcommand{\epsilon}{\varepsilon}
\renewcommand{\phi}{\varphi}

\begin{document}

\title{Trees within trees: Simple nested coalescents}

\maketitle

\begin{center}

 {\normalsize Airam Blancas$^{1}$, Jean-Jil Duchamps$^{2,3}$, Amaury Lambert$^{2,3}$, Arno Siri-J\'egousse$^{4}$}\\
\noindent {\small \it
$^{1}$ Institut für Mathematik, Goethe-Universität, Frankfurt, Germany;\\
$^{2}$ Laboratoire de Probabilit\'es, Statistique et Mod\'elisation (LPSM), Sorbonne Universit\'e, CNRS UMR 8001, Paris, France; \\
$^{3}$ Center for Interdisciplinary Research in Biology (CIRB), Coll\`ege de France, CNRS UMR 7241, INSERM U1050, PSL Research University, Paris, France; \\
$^{4}$ Instituto de Investigaciones en Matemáticas Aplicadas y Sistemas (IIMAS), Universidad Nacional Aut\'onoma de M\'exico (UNAM), Mexico City, Mexico.
}

\end{center}
\medskip

\begin{abstract}{ 
We consider the compact space of pairs of nested partitions of $\N$, where by analogy with models used in molecular evolution, we call ``gene partition'' the finer partition and ``species partition'' the coarser one.
We introduce the class of nondecreasing processes valued in nested partitions, assumed Markovian and with exchangeable semigroup. These processes are said simple when each partition only undergoes one coalescence event at a time (but possibly the same time). Simple nested exchangeable coalescent (SNEC) processes can be seen as the extension of $\Lambda$-coalescents to nested partitions. We characterize the law of SNEC processes as follows. In the absence of gene coalescences, species blocks undergo $\Lambda$-coalescent type events and in the absence of species coalescences, gene blocks lying in the same species block undergo i.i.d. $\Lambda$-coalescents. Simultaneous coalescence of the gene and species partitions are governed by an intensity measure $\nu_s$ on $(0,1]\times {\mathcal M}_1 ([0,1])$ providing the frequency of species merging and the law in which are drawn (independently) the frequencies of genes merging in each coalescing species block.  
As an application, we also study the conditions under which a SNEC process comes down from infinity.}
\end{abstract}



\vspace{0.5cm}


\section{Introduction} \label{sec:intro}
In the framework of population biology, one can see asexual organisms, but also DNA sequences or even species, as replicating particles. The genealogical ascendance of co-existing replicating particles can always be represented by a tree whose tips are labelled by the names of these particles \cite{Lam08, Lam17, SS03}. Even if species are not strictly speaking replicating particles, ancestral relationships between species are also usually represented by a tree whose nodes are interpreted as \emph{speciation} events, i.e., the emergence of two or more species from one single species. 
The inference of the so-called \emph{gene tree} of contemporary DNA sequences from their comparison has a decade-long history. It is considered as a field in its own right, called \emph{molecular phylogenetics} \cite{F04, NK00}, which relies heavily on the theory of Markov processes. (This can be misleading, but the \emph{species tree}, much more often than the gene tree, is called a \emph{phylogeny}.)

When one type of replicating particle is physically embedded in another type of particle, like a virus in its host, their common history can be depicted as a \emph{tree within a tree} \cite{D97, M97, PC98}: tree of dividing parasites inside the tree of dividing hosts, tree of paralogous genes (i.e.\ distinct DNA segments resulting from gene duplication and coding for similar functions) inside the gene family tree, gene tree inside the species tree.
In many such cases, biologists are more interested in the coarser tree rather than in the finer tree.
Typically, the finer tree is a gene tree and is inferred thanks to methods developed in molecular phylogenetics. One of the current methodological challenges in quantitative biology is to devise fast statistical algorithms able to also infer the coarser tree. When the genes are sampled from infecting pathogens of the same species (Influenza, HIV...), the coarser tree is the epidemic transmission process \cite{GPG04, VKB13}. When the genes are sampled from (any kind of) different species, the coarser tree is the \emph{species tree} \cite{PC97, HD09, STD14}. It is often required to use several gene trees nested in the same species tree to infer the latter.

In terms of stochastic modeling, the standard strategy is to  define the two nested trees in a hierarchical model referred to as the \emph{multispecies coalescent model} \cite{R02, DR09} (see also \cite{Eth2011,Ber2009} for recent surveys on general coalescent theory and applications to population genetics).
First, the species tree is fixed or drawn from some classic probability distribution (e.g., pure-birth process stopped at some fixed time, viewed as present time).
Second, each gene sequence is assigned to the contemporary species it is (supposed to be) sampled from. Recall that each contemporary species is in correspondence with a tip of the species tree.
Third, conditional on the species tree, each gene lineage can then be traced backwards in time inside the species tree, starting from the tip species harboring it and traveling through its ancestral species successively.
In addition, gene lineages are assumed to coalesce according to the \emph{censored Kingman coalescent} \cite{K82}, i.e., each pair of lineages \emph{lying in the same species} independently coalesces at constant rate. 

In the case when the species tree is also distributed as a Kingman coalescent, the former two-type coalescent process is a Markov process as time runs backward, that we call the \emph{nested Kingman coalescent} (or `Kingman-in-Kingman') \cite{LS,airam_arno,dawsonMultilevel2018}.
Our goal here is to display a much richer class of Markov models for trees within trees, called \emph{simple nested exchangeable coalescent} (SNEC) processes, where multiple species lineages can merge into one single species lineage, and where simultaneously, within those merging species, multiple gene lineages can merge into one single gene lineage.
To make this more precise, we show in the next display some valid and invalid coalescence events from an initial state where six genes, labeled from 1 to 6, are grouped by pairs in three species lineages.
We represent this situation in the next display by a pair of partitions $\binom{\pi^{s}}{\pi^{g}}$, as in the left-hand side of the display.
Event \eqref{transok} is valid because the first two species merge and simultaneously, \emph{within} these species, genes labeled 1, 2 and 3 coalesce.
On the contrary, event \eqref{transnotok1} is not a valid transition because there are two distinct gene coalescences (1 with 2, and 3 with 4), which is proscribed, and event \eqref{transnotok2} is not valid because the gene coalescence (5 with 6) is outside the species coalescence.

\begin{align}
\Bigg( \!\!\! \mbox{
  \begin{tabular}{l@{~}l@{~}l@{~}l@{~}l@{~}l}
  \{ 1, & 2 \}\{ & 3, &4 \}\{ & 5, & 6 \} \\
  \{ 1 \}\{& 2 \}\{& 3 \}\{& 4 \}\{& 5 \}\{& 6 \}
  \end{tabular} 
} \!\!\! \Bigg)
&\rightarrow 
\Bigg( \!\!\! \mbox{
  \begin{tabular}{l@{~}l@{~}l@{~}l@{~}l@{~}l}
  \{ 1, & 2, & 3, & 4 \}\{ & 5, & 6 \} \\
  \{ 1,\phantom{\}\{}& 2,\phantom{\}\{}& 3 \}\{& 4 \}\{& 5 \}\{& 6 \}
  \end{tabular}
} \!\!\! \Bigg) \tag{A}\label{transok} \\
&\not\rightarrow \Bigg( \!\!\! \mbox{
  \begin{tabular}{l@{~}l@{~}l@{~}l@{~}l@{~}l}
  \{ 1, & 2, & 3, & 4 \}\{ & 5, & 6 \} \\
  \{ 1,\phantom{\}\{}& 2 \}\{& 3,\phantom{\}\{}& 4 \}\{& 5 \}\{& 6 \}
  \end{tabular} 
} \!\!\! \Bigg) \tag{B}\label{transnotok1} \\
&\not\rightarrow \Bigg( \!\!\! \mbox{
  \begin{tabular}{l@{~}l@{~}l@{~}l@{~}l@{~}l}
  \{ 1, & 2, & 3, & 4 \}\{ & 5, & 6 \} \\
  \{ 1 \}\{& 2 \}\{& 3 \}\{& 4 \}\{& 5,\phantom{\}\{}& 6 \}
  \end{tabular} 
} \!\!\! \Bigg) \tag{C}\label{transnotok2}
\end{align}

In brief, SNEC processes are the generalization of \emph{$\Lambda$-coalescents} to processes valued, not in partitions of $\N$, but in pairs of nested partitions of $\N$.
The class of $\Lambda$-coalescents \cite{sagitovGeneral1999, pitmanCoalescents1999}, for which only one coalescence event can occur at a time, is a subclass of Markov, exchangeable processes with possibly non-binary nodes, called \emph{$\Xi$-coalescents}, where several coalescence events can be simultaneous \cite{bertoinRandom2006, schweinsbergCoalescents2000}. 

Non-binary nodes in species trees can be interpreted as \emph{unresolved nodes} (a sequence of binary nodes following each other too closely in time for their order to be inferred correctly) or \emph{radiation} events (periods of frequent speciations due to the opening of new ecological opportunities that can be exploited by different, new species). 
In gene trees, non-binary nodes are increasingly recognized as a conspicuous sign of natural selection both by biologists \cite{LT14, MHA17} and by mathematicians and physicists \cite{BBS2013,DS05, BD13, NH13, DWF13, S17} ; it is also well understood that non-binary nodes could be consequences of bottlenecks as well as large variance in offspring distributions \cite{EW2006,Sch2003}.
The class of SNEC processes includes all these features.
They can distinguish unresolved nodes (sequence of stochastically close, binary coalescences) from radiations (multiple merger in the species tree).
Under the interpretation of non-binary nodes as a result of natural selection, SNEC processes can model the appearance of alleles responsible for positive selection (multiple merger in the gene tree) or for divergent adaptation (multiple merger simultaneously in the gene tree and in the species tree). 

From a mathematical point of view as well, SNEC processes open up the door to many possible new investigations. For example some of us are currently studying the speed of coming down from infinity of SNEC processes \cite{LS,airam_arno} as well as similar extensions \cite{D} to fragmentation processes \cite{bertoinRandom2006}. It will be interesting to investigate how the nested trees generated by SNEC processes can be cast in the frameworks of multilevel measure-valued processes \cite{airam,dawsonMultilevel2018} and flows of bridges \cite{BLG03, BLG06} as well as of exchangeable combs \cite{FRLS,  L18}.
It would also be natural to study the extension of $\Xi$-coalescents 
to nested partitions. 

\paragraph{Organization of the article.}
In Section \ref{sec:statement}, we introduce some notation, and give examples of nested coalescent processes whose distributions are characterized by four parameters.
Section \ref{sec:snec} formally defines our object of study, the SNEC processes.
We prove our main result in Section \ref{sec:proof}, and show in Section \ref{sec:poisson} how SNEC processes can be constructed from a collection of Poisson point processes.
Finally, Section \ref{sec:CDI} gives a necessary and sufficient condition under which SNEC processes come down from infinity.

\section{Statement of results and notation}
\label{sec:statement}

\subsection{Statement of results and examples}

An exchangeable partition is a random partition of $\N$ whose law is invariant by permutations of $\N$ (with finite support). A $\Lambda$-coalescent is a Markov process valued in the exchangeable partitions of $\N$ typically starting from the partition $\mathbf{0}_\infty$ of $\N$ into singletons, and such that only one coalescence event can occur at a time. The generator of a $\Lambda$-coalescent $\mathcal R=(\mathcal R (t), t\ge 0)$ is characterized by a $\sigma$-finite measure $\nu$ on $(0,1]$ called the coagulation measure and a non-negative real number $a$ called the Kingman coefficient. Then $\mathcal R$ can be constructed from a Poisson point process as follows.

For $x\in (0,1]$, let $P_x$ denote the law of a sequence of i.i.d. Bernoulli$(x)$ r.v.'s and define 
\[
P := \int_{(0,1]}\nu(dx) P_x
\]
Also define ${\tt K}_{i,i'}$ the (Dirac) law of the sequence with only zero entries except a 1 at positions $i$ and $i'$ and set
\[
{\tt K}:=\sum_{1\le i<i'} {\tt K}_{i,i'}
\]
Finally, let $M$ be a Poisson point process with intensity measure $dt\otimes (P+a {\tt K})$. 
Roughly speaking, at each atom $(t, (X_i,i\ge 1))$ of $M$, $\mathcal R(t)$ is obtained from $\mathcal R(t-)$ by merging exactly the $i$-th block of $\mathcal R(t-)$ together, for all $i$ such that $X_i=1$. The rigorous description is given through restrictions of $\mathcal R$ to $[n]:=\{1,\ldots, n\}$ and by applying Kolmogorov extension theorem. See  \cite{bertoinRandom2006} for details. Note that for this description to apply (i.e., for restrictions of $\mathcal R$ to $[n]$ to have positive holding times), one needs the coagulation measure to satisfy
\begin{equation}
\label{eqn:condLambda}
\int_{(0,1]}x^{2}\,\nu(dx) <\infty.
\end{equation}
The finite measure $x^{2}\,\nu(dx)$ is usually denoted $\Lambda(dx)$, hence the name $\Lambda$-coalescent. 
\medskip

We can now draw the parallel with the results obtained in this paper. We want to define a Markov process $\mathcal{R}=\left( (\mathcal{R}^s(t), \mathcal{R}^g(t)),\; t\geq 0 \right)$ valued in exchangeable bivariate, nested partitions of $\N$, in the sense that the \emph{gene partition} $\mathcal{R}^g(t)$ is finer than the \emph{species partition} $\mathcal{R}^s(t)$ for all $t$ a.s.

We now have to allow for coalescences in both the gene partition and the species partition. To this aim, we will consider a doubly indexed array of 0's and 1's $\mathbf{Z}={(\mathbf{X}, (\mathbf{Y}_i, \, i\geq 1))}=(X_i, Y_{ij}, \, i,j\geq 1)$.
The goal is to give a characterization and a Poissonian construction of $\mathcal R$ under the assumptions that the semigroup of $\mathcal R$ is exchangeable and that both $\mathcal{R}^s$ and $\mathcal{R}^g$ undergo only one coalescence at a time (but possibly the same time), as detailed in forthcoming Definition \ref{snec}. Roughly speaking, and similarly as previously, $X_i$ will determine whether the $i$-th species block participates in the coalescence in the species partition $\mathcal{R}^s$, and $Y_{ij}$ whether the $j$-th gene block of the $i$-th species block participates in the coalescence in the gene partition $\mathcal{R}^g$.

Let us start with the Kingman-type coalescences.
Let ${\tt K}^s_{i,i'}$ be the (Dirac) law of the array $\mathbf{Z}$ with only zero entries except $X_i = X_{i'}=1$ and let ${\tt K}^g_{i;j,j'}$ be the (Dirac) law of the array $\mathbf{Z}$ with only zero entries except $X_i =Y_{ij}= Y_{ij'}=1$. Finally, define 
\[
{\tt K}^s=\sum_{1\le i<i'}{\tt K}^s_{i,i'}\quad \mbox{ and } \quad {\tt K}^g= \sum_{1\le i}\sum_{1\le j<j'}{\tt K}^g_{i;j,j'}
\]
Let us carry on with multiple gene mergers without simultaneous species coalescences.
Let $x\in (0,1]$ and $i\in \N$. Let $P^g_{i, x}$ be the distribution of the array $\mathbf{Z}$ with only zero entries except at row $i$, where $X_i=1$ and the $(Y_{ij}, \,j\ge 1)$ are i.i.d.\ Bernoulli$(x)$ r.v.'s. Let us define 
\[
P^g_x := \sum_{i\ge 1}P^g_{i, x} 
\]
Finally, let us consider multiple species mergers, with possible simultaneous gene mergers.
Let $x\in (0,1]$ and $\mu\in \mathcal{M}_1 ([0,1])$. Let $(X_i,\,i\geq 1)$ be a sequence of i.i.d.\ Bernoulli$(x)$ r.v.'s and let $(Q_i,\,i\geq 1)$ be an independent sequence of i.i.d. r.v.'s of $[0,1]$ with distribution $\mu$. Then for each $i\ge 1$, conditional on $X_i$ and $Q_i$, let $(Y_{ij}, \, j\geq 1)$ be an independent sequence of i.i.d.\  Bernoulli$(Q_i)$ r.v.'s. if $X_i=1$ and the null array otherwise.
Let us write $P^s_{x,\mu}$ for the distribution of the array $\mathbf{Z}$ thus defined.

Our main result is that for any simple nested exchangeable coalescent (SNEC) process $\mathcal R$,
there are
\begin{itemize}
  \item two non-negative real numbers $a_s$ and $a_g$;
  \item a $\sigma$-finite measure $\nu_g$ on $(0,1]$;
  \item a $\sigma$-finite measure $\nu_s$ on $(0,1]\times\mathcal M_1([0,1])$,
\end{itemize}
such that $\mathcal{R}$ can be constructed from a Poisson point process  $M$ with intensity $dt \otimes \nu(d\mathbf{Z})$ where
\[
\nu :=  a_s{\tt K}_s + a_g {\tt K}_g+ \int_{(0,1]}\nu_g (dx)\, P^g_{x} +\int_{(0,1]\times \mathcal{M}_1 ([0,1])}\nu_s(dx,d\mu)\, P^s_{x,\mu}.
\]

Similarly as explained previously, at each atom $(t, \mathbf{Z})$ of $M$, the double array $\mathbf{Z}$ prescribes which blocks have to merge at time $t$. For the finite restrictions of $\mathcal R$ to have positive holding times, the measures $\nu_s$ and $\nu_g$ are required to satisfy the forthcoming conditions \eqref{eq:condspecies} and \eqref{eq:condgenes}  respectively, which are the analogs to \eqref{eqn:condLambda}.

Note that coagulations of the Kingman type cannot occur simultaneously in the species partition and in the gene partition.

We now give a couple of examples of SNEC processes.

If $\nu_s (dp, d\mu) = \nu_s'(dp)\, \delta_{\delta_0}(d\mu)$, species and genes never coalesce simultaneously and the nested coalescent is a multispecies coalescent (see Introduction), where the species tree is given by the $\Lambda$-coalescent with coagulation measure $\nu_s'$ and Kingman coefficient $a_s$, while the genes in the same species block undergo independent $\Lambda$-coalescents with coagulation measure $\nu_g$ and Kingman coefficient $a_g$. In particular, when $\nu_s'$ and $\nu_g$ are zero, the SNEC process is a nested Kingman coalescent (Kingman-in-Kingman). 

Whenever $\nu_s$ is not under the form $\nu_s (dp, d\mu) = \nu_s'(dp)\, \delta_{\delta_0}(d\mu)$, species blocks and gene blocks can coalesce simultaneously.
For example if $\nu_s (dp, d\mu) = \nu_s'(dp)\, \delta_{\delta_x}(d\mu)$ for $x\in (0,1]$, at each species coalescence event, a proportion $x$ of gene blocks contained in the species blocks participating in the coalescence event, are simultaneously merged together. 
In particular, if $x=1$, the gene tree coincides with the species tree on lineages situated after a species coalescence event. 
Recall that there are conditions (see \eqref{eq:condspecies}) for $\nu_s$ to be a correct SNEC measure, which in this case translate to
\begin{gather*}
\int_{(0,1]}\nu_s' (dp)\,p^2<\infty \quad \text{ and } \quad \int_{(0,1]}\nu_s'(dp)\,p \,x^2 <\infty,
\end{gather*}
which is simply equivalent to
\[
\int_{(0,1]}\nu_s'(dp)\,p < \infty.
\]
Otherwise the simplest sort of measure $\nu_s$ can be obtained by parameterizing its second component $\mu$, for example if $\mu$ is a Beta distribution $\mu_{a,b}(dq) = c_{a,b}\,q^{a-1}(1-q)^{b-1}\, dq$, where $a, b>0$ and $c_{a,b} = \frac{\Gamma(a+b)}{\Gamma(a)\,\Gamma(b)}$, we can consider $\nu_s$ under the form
\[
\nu_s(dp,d\mu) = \nu_s' (dp, da, db)\, \delta_{\mu_{a,b}}(d\mu).
\]
In this case, the condition \eqref{eq:condspecies1} reads
\begin{equation*}
\int_{(0,1]\times (0,\infty)\times(0,\infty)}\nu_s' (dp, da, db)\,p^2<\infty,
\end{equation*}
and \eqref{eq:condspecies2} becomes
\begin{equation*}
\int_{(0,1]\times (0,\infty)\times(0,\infty)}\nu_s'(dp, da, db)\,p \int_{[0,1]}c_{a,b}\,q^{a+1}(1-q)^{b-1}\, dq<\infty,
\end{equation*}
which can be rewritten
\[
\int_{(0,1]\times (0,\infty)\times(0,\infty)}\nu_s' (dp, da, db)\,\frac{pa(a+1)}{(a+b)(a+b+1)}<\infty.
\]
Note that the idea to use a Beta distribution here is inspired by the $\Lambda$-coalescent setting \cite{pitmanCoalescents1999}, where Beta distributions appear as natural candidates for the parametrization of the measure $\Lambda$, as the coalescence rate of each $k$-tuple of blocks among a total number of $b$ blocks is expressed in the form
\[
\int_{0}^{1}x^{k-2}(1-x)^{b-k}\Lambda(dx).
\]

\subsection{Notation}

For any $n\in\bar{\mathbb{N}}:=\N\cup\{+\infty\}$, let $\cP_n$ be the set of partitions of $[n]$ .
A partition $\pi$ is called \emph{simple} if at most one of its non-empty blocks is not a singleton. 
We denote the set of simple partitions of $[n]$ by $\mathcal{P}_n^\prime$, that is,
\[
\mathcal{P}_n^\prime= \{ \pi\in\mathcal{P}_n ,\; \text{Card}\{ i,\; | \pi_i | >1\} \leq 1 \}
\]
where $\pi_1,\pi_2,\dots$ denote the blocks of $\pi$ ordered by their least element
and $|\pi_i|$ stands for the number of elements in the block $\pi_i$.
Recall that a partition $\pi$ can be viewed as an equivalence relation, in the sense that $i \overset{ \pi}{\thicksim}j$ if and only if $i$ and $j$ belong to the same block of the partition $\pi$. 
If $\pi^g$ and $\pi^s$ belong to $\cP_n$, we will say that the bivariate partition $\pi=(\pi^s,\pi^g)$ is \emph{nested}  (or equivalently that $\pi^g$ is finer than $\pi^s$) when 
\[
i \overset{ \pi^g}{\thicksim}j\ \Longrightarrow\ i \overset{ \pi^s}{\thicksim}j.
\] 
Note that this is defines a natural partial order on $\mathcal{P}_n$, and we can write $\pi^{g}\preceq \pi^{s}$ if $(\pi^{g},\pi^{s})$ is nested.
The set of nested partitions of $[n]$ is denoted in the sequel by $\cN_n$.
We will sometimes use the notation $\mathbf{1}_n:=\{[n]\}$ for the coarsest partition of $[n]$, and $\mathbf{0}_n := \{\{1\}, \{2\},\ldots\}$ for the finest partition of $[n]$.
\begin{example}\label{ex1}
  An example of nested partition of $\{1,2,\ldots,10\}$ is given by
  \begin{align*}
  \pi^s & = \big\{\{  1,5,7 \} , \;\{ 2,4,8,10 \} ,\; \{ 3,6,9 \}\big\}\\
  \pi^g & = \big \{ \{ 1 \} , \;\{ 2,4 \} , \; \{ 3\} , \; \{ 5,7\}, \; \{ 6,9 \} ,  \;\{ 8 \}, \; \{ 10 \} \big \}.
  \end{align*}
\end{example}
The notation $(\pi^s,\pi^g)$ owes to our modeling inspiration (see Introduction) where gene lineages are enclosed into species lineages. 

Notation related to and properties of $\mathcal{P}_n$ can naturally be extended to the framework of bivariate partitions. For the sake of completeness we specify here the ones we will use repeatedly. 
The number of non-empty blocks of a bivariate partition $\pi=(\pi_1,\pi_2) \in \cP_{n_1}\times\cP_{n_2}$ is merely $|\pi|:=(|\pi_1|,|\pi_2|)$. 
If $m_1<n_1$ and $m_2<n_2$, we write $\pi_{| m_1\times m_2}$ for the restriction of $\pi$ to $\cP_{m_1}\times\cP_{m_2}$, that is, $\pi_{ | m_1\times m_2}=(\pi_{1\,|m_1},\pi_{2\,|m_2}) $. If $m\leq\min(n_1,n_2)$, we will write $\pi_{|m}:=\pi_{|m\times m}$ for its restriction to $\mathcal{P}^2_{m}:=\cP_m\times\cP_m$. 
A sequence $\pi^{ (1) } ,\pi^{ (2) },... $ of elements of $\mathcal{P}_1^2, \mathcal{P}_2^2,... $ is called \emph{consistent} if for all integers $k' \leq k$, $\pi^{ (k') }$ coincides with the restriction of $\pi^{(k)}$ to $[k']^2$. Moreover, a sequence of partitions $(\pi^{(n)}:n \in\mathbb{N})$ is consistent if and only if there exists $\pi\in\mathcal{P}^2_\infty$ such that $\pi_{ |{ n } }=\pi^{(n)}$ for every $n\in\mathbb{N}$.


Given a nested partition we can use the coagulation operator Coag (more details in Chapter 3 in Bertoin \cite{bertoinRandom2006}) to write the species partition in terms of the labels of the gene partition.
Recall that if $\pi\in\cP_n$ and $\tilde \pi\in\cP_m$ with $m\geq|\pi|$, then define $\pi^\prime=\text{Coag}(\pi, \tilde\pi)$  as the partition of $\cP_n$ such that
\[
\pi^{\prime }_j = \bigcup_{ i \in \tilde{\pi}_j } \pi_i.
\]
For every $n \in \bar{\mathbb{N}}$, let $\pi=(\pi^s, \pi^g)$ be an element of $\mathcal{N}_n$ and write $m= | \pi ^g |$. The unique partition $\bar{\pi}\in\mathcal{P}_{m}$ such that $\pi^s=\text{Coag}(\pi^g,\bar{\pi})$ is called the \emph{link} partition of $\pi$. We sometimes say that $\pi$ is linked by $\bar{\pi}$.
To illustrate the previous definition, observe that the nested partition defined in Example \ref{ex1} has link partition $ \bar{\pi}= \{\{ 1,4\}, \{ 2,6,7\}, \{ 3,5 \}\}$.

We can next get a partition of $ \mathcal{P}_{n_1}\times\mathcal{P}_{n_2}$ through the coagulation of two pairs of  partitions. More precisely, if $(\pi^1,\tilde{\pi}^1 )\in \mathcal{P}_{n_1}\times\mathcal{P}_{n_1^\prime}$ and $(\pi^2,\tilde{\pi}^2 ) \in \mathcal{P}_{n_2}\times\mathcal{P}_{n_2^\prime}$ with $n_1^\prime\geq|\pi^1|$ and $n_2^\prime\geq|\pi^2|$, then $ ( \text{Coag}( \pi^1 , \tilde{\pi}^1 ) , \text{Coag}( \pi^2, \tilde{\pi}^2 ) )$ is well defined and it is an element of $\mathcal P_{n_1}\times \mathcal P_{n_2}$.
{If we denote $\pi=(\pi^1,\pi^2)$ and $\tilde{\pi}=(  \tilde{\pi}^1, \tilde{\pi}^2 )$
  we will say that the pair $(\pi,\tilde\pi)$ is \emph{admissible} and denote}
the latter operation by $\text{Coag}_2( \pi,\tilde\pi ) $. In the following we will sometimes call the partition $\tilde\pi$ as the \emph{recipe} partition.

In the sequel, we are interested in the coagulation of a nested partition, say $\pi=(\pi^s,\pi^g)$, with a pair of simple partitions $\tilde{\pi}=( \tilde{\pi}^s, \tilde{\pi}^g ) $. Nevertheless, we should observe that the resulting partition, $\text{Coag}_2(\pi,\tilde\pi)$ is not necessarily nested. For instance, if we coagulate the partition $\pi$ of Example \ref{ex1}, with  $\tilde{\pi}^{ s}=\{ \{ 1,2 \}, \{3  \}\}$, and $\tilde{\pi}^{ g}=\{\{ 1,3 \}, \{2 \}, \{ 4\}, \{ 5\}, \{ 6\}, \{ 7 \}\}$ then $\text{Coag}(\pi^g,\tilde\pi^g)$ is not nested in $ \text{Coag}(\pi^s,\tilde\pi^s)$.
In order to maintain the nested property while coagulating a nested partition we need to watch out the way the gene blocks do merge together and if they respect the species structure.
To this end, for any $n\in \bar{\mathbb{N}}$ and $\pi\in \mathcal{N}_n$, we can define the set $\widetilde{\mathcal{P}}(\pi) \subset (\mathcal{P}'_n)^2$ of simple recipe partitions permitting a consistent merger of species and genes, i.e.
\[
\widetilde{\mathcal{P}}(\pi)=\Big\{\tilde{\pi}=(\tilde{\pi}^s,\tilde{\pi}^g)\in(\mathcal{P}_n')^2, \; i \overset{\widetilde{\pi}^g}{\sim} j \implies i\overset{\bar\pi}{\sim}j, \text{ or } k\overset{\widetilde{\pi}^s}{\sim}l, \text{ where } \pi^{g}_i \subset \pi^{s}_k \text{ and }\pi^{g}_l \subset \pi^{s}_l\Big\},
\]
where $\bar\pi$ denotes as usual the link partition of $\pi$.
Simply put, $\widetilde{\mathcal{P}}(\pi)$ is the subset of $(\mathcal{P}_n')^2$ such that
\[
\tilde{\pi}\in\widetilde{\mathcal{P}}(\pi) \iff \text{Coag}_2(\pi, \tilde{\pi})\in \mathcal{N}_n.
\]
Finally the natural partial order on partitions can be extended to bivariate partitions by defining $(\pi^{1,s},\pi^{1,g}) \preceq (\pi^{2,s},\pi^{2,g}) \iff \pi^{1,s} \preceq \pi^{2,s}$ and $\pi^{1,g} \preceq \pi^{2,g}$.
This partial order allows us to see coalescent processes as nondecreasing processes in the space of nested partitions.

\section{Simple nested exchangeable coalescents} \label{sec:snec}
In the aim to describe the joint dynamics of the species and gene partitions, we will now define a nondecreasing process with values in the nested partitions, called \emph{nested coalescent process}.
In this work we are only interested in \emph{simple} nested coalescents in the sense that at any jump event, called coalescence event, all blocks undergoing a modification merge into one single block.
Simple exchangeable coalescent processes were first introduced independently by Pitman \cite{pitmanCoalescents1999} and Sagitov \cite{sagitovGeneral1999}, and are usually called in the literature $\Lambda$-coalescents (see Introduction). 
Here we use the term \emph{simple} as in \cite{bertoinRandom2006}, to denote the analog of a $\Lambda$-coalescent process in the case of (nested) bivariate partitions.

Note that for any partition $\pi\in \mathcal{P}_{\infty}$ and any \emph{injection} $\sigma:\N\to\N$, there is a partition $\sigma(\pi)$ defined by
\[ i\overset{\sigma(\pi)}{\thicksim} j \iff \sigma(i)\overset{\pi}{\thicksim} \sigma(j). \]
For bivariate partitions we define in the same way $\sigma(\pi^s, \pi^g) := (\sigma(\pi^s),\sigma(\pi^g))$.
For random partitions, exchangeability is usually defined as invariance under the action of permutations $\sigma:\N\to\N$.
Here, to avoid degenerate processes we will define our processes as being invariant under the action of all injections $\sigma:\N\to\N$.
Indeed, by making this assumption we avoid dependence on, for instance, the total number of blocks of the partition.
An example of what we consider here a degenerate process with values in $\mathcal{P}_\infty$ would be a modified Kingman coalescent where any pair of blocks merge at rate $a=a(n)$, a function of $n$ the total number of blocks.
While this process would be invariant under permutations of $\N$, it is in general not invariant under injections, as their action can change the total number of blocks in a partition of $\N$.
Furthermore, given $(\Pi(t),\,t\geq 0)$ such a process and $n$ an integer, the restriction $(\Pi(t)_{|n}, \,t\geq 0)$ would not be a Markov process, as the jump rates of $\Pi(t)_{|n}$ depend on the whole partition $\Pi(t)$.
Invariance under injections ensures us that processes can be consistently defined, i.e.\ that $(\Pi(t)_{|n}, \,t\geq 0)$ will always be a Markov process.
It will also be useful in forthcoming proofs to consider invariance under injections rather than only permutations.

Since we consider processes with values in the space $\mathcal{P}_\infty$, let us endow it with the natural topology generated by the sets of the form $\{\pi' \in \mathcal{P}_\infty,\, \pi'_{|n} = \pi \}$ for $n\in\N$ and $\pi\in\mathcal{P}_n$.
It is readily checked that this topology is metrizable and makes $\mathcal{P}_\infty$ compact.
Also, note that the product topology on $\mathcal{P}_\infty^2$, and that induced on $\mathcal{N}_\infty$ also makes them compact.


\begin{definition}\label{snec}
  
  Let $\mathcal{R}:=\left( (\mathcal{R}^s(t), \mathcal{R}^g(t)) ,\; t\geq 0 \right)$ be a càdlàg Markov process with values in $\mathcal{P}^2_\infty$. 
  This process is called a \emph{simple nested exchangeable coalescent}, SNEC for short, if 
  \begin{enumerate}[i)]
    \item For any $t\geq0$, $\mathcal{R}(t)$ is nested;
    
    \item The process $( \mathcal{R}(t) ,\; t \geq 0  )$ evolves with simple coalescence events, that is for any time $t\geq 0$ such that $\mathcal{R}(t-)\neq \mathcal{R}(t)$, {there is a random bivariate partition $\widetilde{\mathcal{R}}(t) = (\widetilde{\mathcal{R}}^s(t), \widetilde{\mathcal{R}}^g(t))$ taking values in $\widetilde{\mathcal{P}}(\mathcal{R}(t-))$ such that}
    \[
    {\mathcal{R}(t) = \text{Coag}_2( \mathcal{R}(t-)  , \widetilde{\mathcal{R}}(t));} 
    \]
    
    \item \label{we} The semigroup of the process $( \mathcal{R}(t) ,\; t \geq 0  )$ is exchangeable, in the sense that for any $t,t'\geq 0$ and any \emph{injection} $\sigma:\N\to\N$,
    \begin{equation}\label{eq:defSNECexch}
    \left (\sigma(\mathcal{R}(t+t')) \mid \mathcal{R}(t) = \pi\right ) \ed \left (\mathcal{R}(t+t') \mid \mathcal{R}(t) = \sigma(\pi)\right ).
    \end{equation}
  \end{enumerate}
\end{definition}

To start the analysis of SNEC processes we would like to make some observations related to Definition \ref{snec}.
First note that $\mathcal{R}$ is a $\mathcal{N}_\infty$-valued process such that for every $t,t^\prime\geq 0$, the conditional distribution of $\mathcal{R}(t+t^\prime) $ given $\mathcal{R}(t) = \pi$ is the law of $\text{Coag}_2( \pi  , \tilde{\pi})$, where $\tilde{\pi}\in \widetilde{\mathcal{P}}( \pi)$, hence the law of $\tilde{\pi}$ depends on $t'$ but also on $\pi$.  
{Also, it will be clear from our main result (see Theorem \ref{thm:charac}) that $(\mathcal{R}^s (t),\; t\geq 0)$ is an exchangeable coalescent}, however $(\mathcal{R}^g (t),\;t\geq0)$ is not a Markov process in general, because the distribution of $\mathcal{R}^g(t+t')$ may depend on $\mathcal{R}^s (t)$.

We now turn to investigate the transitions of the restrictions of a SNEC to finite partitions, which relies on the following lemma.

\begin{lemma}[Projective Markov property] \label{lem:projMarkov}
  Let $\mathcal{R}=(\mathcal{R}(t) ,\; t\geq 0)$ be a process with values in $\mathcal{N}_\infty$ and for every integer $n$, write $\mathcal{R}_{| {n} }=(\mathcal{R}_{| {n} }(t) ,\; t\geq 0)$ for its restriction to $\mathcal{N}_n$.
  Then $\mathcal{R}$ is a SNEC in $\mathcal{N}_\infty$ if and only if for all $n\in\N$, $\mathcal{R}_{| {n} }$ 
  is a continuous-time Markov chain on the space $\mathcal{N}_n$ satisfying the analog of statements \textit{i)} -- \textit{iii)} of Definition \ref{snec}, namely:
  \begin{enumerate}[i)]
    \item For all $t\geq 0$, $\mathcal{R}_{| {n} }$ is nested;
    \item For $\rho,\pi\in\mathcal{N}_n$, the rate from $\rho$ to $\pi$ is zero if $\pi$ can not be obtained from a simple coalescence event;
    \item The Markov chain $( \mathcal{R}_{|n}(t) ,\; t \geq 0  )$ is exchangeable, in the sense that for any $t,t'\geq 0$, $\rho,\pi\in\mathcal{N}_n$ and $\sigma$ \emph{permutation} of $n$, the rate from $\rho$ to $\pi$ is equal to that from $\sigma(\rho)$ to $\sigma(\pi)$.
  \end{enumerate}
\end{lemma}

\begin{proof} Let $\mathcal{R}$ be a SNEC in $\mathcal{N}_\infty$ and let $n \in \N$. Let us prove that $\mathcal{R}_{| {n} }$ satisfies the claimed properties.
  Let  $\rho\in\mathcal{N}_{n}$.
  Pick  $\rho^{\star}\in\mathcal{N}_{\infty}$ such that $\rho^{\star}_{|n}=\rho$, and which contains an infinite number of species blocks, each of which containing an infinite number of gene blocks, each of them being an infinite subset of $\N$.
  Now for any $\rho'\in\mathcal{N}_\infty$ such that $\rho'_{|n}=\rho$, there is an injection $\sigma:\N\to\N$ such that $\sigma(\rho^{\star})=\rho'$ and such that $\sigma_{|[n]}=\text{id}_{[n]}$, so for any $t,t'\geq 0$,
  \begin{align*}
  (\mathcal{R}_{|n}(t+t') \,\mid\, \mathcal{R}(t) =\rho') &\ed (\mathcal{R}_{|n}(t+t') \,\mid\, \mathcal{R}(t) =\sigma(\rho^{\star}))\\
  &\ed ( \sigma(\mathcal{R})_{|n}(t+t') \,\mid\, \mathcal{R}(t) =\rho^{\star}) \\
  &\ed ( \mathcal{R}_{|n}(t+t') \,\mid\, \mathcal{R}(t) =\rho^{\star}).
  \end{align*}
  Since this is valid for any $\rho'$ such that $\rho'_{|n}=\rho$, this conditional distribution depends only on $\{\mathcal{R}_{|n}(t)=\rho\}$, which proves that $\mathcal{R}_{|n}$ is a Markov process.
  Now the assumption that $\mathcal{R}$ has càdlàg paths ensures us that the process $\mathcal{R}_{|n}$ stays some positive time in each visited state \textit{a.s.}
  Therefore $\mathcal{R}_{|n}$ is a continuous-time Markov chain. Now statements \textit{i)} -- \textit{iii)} are easily deduced from Definition \ref{snec}.
  
  Conversely, let $\mathcal{R}=(\mathcal{R}(t),\,t\geq 0)$ be a process with values in $\mathcal{N}_\infty$ such that for all $n\in\mathbb{N}$, $\mathcal{R}_{|n}$ is a Markov chain satisfying \textit{i) -- iii)} of the lemma.
  Then \textit{i)} and \textit{ii)} of Definition \ref{snec} follow immediately, and it remains to check that for any injection $\sigma :\N\to\N$, the equality in distribution \eqref{eq:defSNECexch} holds.
  
  Let $\sigma :\N\to\N$ be an injection and fix $n\in \N$.
  Define $N = \max\{ \sigma(1),\sigma(2), \ldots, \sigma(n)\}$, and consider $\widetilde{\sigma}:[N]\to[N]$ a permutation such that for all $1\leq i \leq n$, $\widetilde{\sigma}(i)=\sigma(i)$.
  For instance, one can define inductively for $n+1\leq i\leq N$,
  \[ \widetilde{\sigma}(i) := \min ([N]\setminus\{\sigma(1),\sigma(2), \ldots, \sigma(i-1)\}). \]
  Now notice that for any $t\geq 0$ and any $\pi\in\mathcal{N}_\infty$,
  \[ \sigma(\pi)_{|n} = \widetilde{\sigma} (\pi_{|N} )_{|n}, \]
  which enables us to write, for any $t,t'\geq 0$,
  \begin{align*}
  (\sigma(\mathcal{R})_{|n}(t+t') \mid \mathcal{R}(t)=\pi) &\ed  (\widetilde{\sigma} (\mathcal{R}_{|N}(t+t') )_{|n} \mid  \mathcal{R}(t)=\pi )\\
  &\ed (\mathcal{R}_{|N}(t+t')_{|n} \mid  \mathcal{R}_{|N}(t)=\widetilde{\sigma}(\pi_{|N}) )\\
  &\ed (\mathcal{R}_{|n}(t+t') \mid  \mathcal{R}_{|n}(t)=\widetilde{\sigma}(\pi_{|N})_{|n} )\\
  &\ed (\mathcal{R}_{|n}(t+t') \mid  \mathcal{R}_{|n}(t)=\sigma(\pi)_{|n} )\\
  &\ed (\mathcal{R}_{|n}(t+t') \mid  \mathcal{R}(t)=\sigma(\pi) ).
  \end{align*}
  The passage to the second line in the last display is a consequence of \textit{iii)} of the lemma, and we used the fact that restrictions are Markov chains, i.e.\ ${(\mathcal{R}_{|n}(t+t') \mid  \mathcal{R}_{|n}(t)=\pi_{|n} ) } \ed { (\mathcal{R}_{|n}(t+t') \mid  \mathcal{R}(t)=\pi)}$.
  Since $n$ is arbitrary in
  \[ (\sigma(\mathcal{R})_{|n}(t+t') \mid \mathcal{R}(t)=\pi) \ed { (\mathcal{R}_{|n}(t+t') \mid  \mathcal{R}(t)=\sigma(\pi) )},\]
  we have shown \eqref{eq:defSNECexch}, concluding the proof.
\end{proof}

This key lemma enables us to give the following first properties of SNEC processes.
\begin{proposition}
  Let $\mathcal{R}$ be a SNEC. 
  \begin{itemize}
    \item If the process $\mathcal{R}$ starts from an exchangeable nested partition $\mathcal{R}(0)$, then for any $t\geq0$, $ \mathcal{R}^g(t)$ and $ \mathcal{R}^s(t)$ are exchangeable partitions.
    
    \item The process $\mathcal{R}$ is a Feller process, so in particular it satisfies the strong Markov property.
    
    \item Conditional on $\mathcal{R}(t)$, if $\bar{\mathcal{R}}(t)$ denotes the link partition of $\mathcal{R}(t)$
    then for any $t,t^\prime\geq0$, the distribution of $\mathcal{R}^g(t+t')$ is the law of $\text{Coag}(\mathcal{R}^g(t), \tilde{\pi}^g )$, where $\tilde{\pi}^g$ is a random partition such that
    $\sigma(\tilde{\pi}^g) \overset{d}{=} \tilde{\pi}^g$ for any permutation $\sigma$ preserving $\bar{\mathcal{R}}(t) $ i.e., such that
    \begin{equation} \label{condnestAL}
    i \overset{ \bar{\mathcal{R}}(t) }{\thicksim} j \Rightarrow \sigma(i)  \overset{  \bar{\mathcal{R}}(t) }{\thicksim}  \sigma(j).
    \end{equation}
  \end{itemize}
\end{proposition}

Another property is that the process $( \mathcal{R}^s(t) ,\; t \geq 0  )$ is a simple exchangeable coalescent process, but we do not prove it at this point as it will be clear from Theorem \ref{thm:charac}.

\begin{proof}
  The first point of the proposition is immediate considering \textit{iii)} of Definition \ref{snec}.
  
  As for the second point, recall that $\mathcal{N}_\infty$ is endowed with the topology generated by the sets of the form $\{\pi\in\mathcal{N}_\infty, \, \pi_{|n} = \widehat{\pi} \}$, for $n\in\N,\,\widehat{\pi}\in\mathcal{N}_n$. It is easy to see that this topology is metrized by $d(\pi,\pi'):=(\sup\{n\in\N,\, \pi_{|n}=\pi'_{|n}\})^{-1}$ (with $(\sup\N)^{-1}=0$) and that $(\mathcal{N}_\infty,d)$ is compact.
  
  We need to show that for any continuous (then bounded) function $f:\mathcal{N}_\infty\to \R$, the function $P_tf:\pi \mapsto \E_{\pi}f(\mathcal{R}(t))$ (where $\E_\pi(\cdot)=\E(\cdot\,|\, \mathcal{R}(0)= \pi )$) is continuous, and that $P_tf(\pi)\to f(\pi)$ as $t\to 0$.
  By definition the process is càdlàg so we have almost surely $f(\mathcal{R}(t))\to f(\mathcal{R}(0))$ so clearly by taking expectations $P_tf(\pi)\to f(\pi)$ as $t\to 0$.
  Now to show that $P_tf$ is continuous, consider $n\in\N$ and let $\{\widehat{\pi}^1, \ldots, \widehat{\pi}^k\}$ be an enumeration of $\mathcal{N}_n$.
  We pick $\pi^1, \ldots, {\pi}^k \in \mathcal{N}_\infty$ such that $\pi^i_{|n} = \widehat{\pi}^i$, and define $\widehat{f}_n:\mathcal{N}_\infty \to \R$ by
  \[ \widehat{f}_n(\pi) = f(\pi^i) \quad \text{ if } \pi_{|n} = \widehat{\pi}^i. \]
  Now since $f$ is continuous on $(\mathcal{N}_\infty,d)$ which is compact, $f$ is uniformly continuous, which means that
  \[ \omega_n := \sup_{\pi\in\mathcal{N}_\infty}\abs{f(\pi)-\widehat{f}_n(\pi)} \to 0 \quad \text{as }n \to \infty. \]
  For $t>0$ and $\pi, \pi' \in\mathcal{N}_\infty$, we have
  \begin{equation}\label{eq:feller}
  \abs{P_tf(\pi)-P_tf(\pi')} \leq \left\abs{\E_{\pi}\widehat{f}_n(\mathcal{R}(t))-\E_{\pi'}\widehat{f}_n(\mathcal{R}(t)) \right} + 2\omega_n.
  \end{equation}
  Now suppose $\pi_{|n} = \pi'_{|n}$.
  Since $\widehat{f}_n$ depends only on $\pi_{|n}$ and by Lemma \ref{lem:projMarkov} the process $\mathcal{R}_{|n}$ has the same distribution under $\P_\pi$ or $\P_{\pi'}$, we have the equality $\E_{\pi}\widehat{f}_n(\mathcal{R}(t))=\E_{\pi'}\widehat{f}_n(\mathcal{R}(t)) $, and plugging that into \eqref{eq:feller}, we get
  \[ \sup\{\abs{P_tf(\pi)-P_tf(\pi')}, \;\pi,\pi'\in\mathcal{N}_\infty, \, \pi_{|n}=\pi'_{|n}\} \leq 2\omega_n \to 0 \quad \text{as } n\to \infty,  \]
  showing that $P_tf$ is continuous.
  
  For the third point of the proposition, $\mathcal{R}^g(t+t')$ is clearly of the form $\text{Coag}(\mathcal{R}^g(t), \tilde{\pi}^g )$, where $\tilde{\pi}=(\tilde{\pi}^s, \tilde{\pi}^g)$ is a random recipe partition whose distribution depends on $\mathcal{R}(t)$ and $t'$.
  Let us show that the conditional distribution of $\tilde{ \pi }$ given $\mathcal{R}(t)$ is invariant under the action of permutations preserving $\bar{\mathcal{R}}(t)$.
  
  Without loss of generality, we can work under the conditioning $\{\mathcal{R}(t)=(\rho,\mathbf{0}_\infty)\}$, where $\rho$ is any partition and $\mathbf{0}_\infty$ is the partition into singletons, so that for all $\pi$, we have $\text{Coag}(\mathcal{R}^g(t), \pi) = \pi$.
  {In particular, note that in this case we have $\mathcal{R}^g(t+t')=\tilde\pi^g$, and $\bar{\mathcal{R}}(t) = \rho$.
    Let $\sigma$ be a permutation such that $\sigma(\rho)=\rho$.}
  The problem then reduces to showing that
  \[ \left (\sigma(\mathcal{R}^g(t+t'))\mid\mathcal{R}(t)=(\rho, \mathbf{0}_\infty)\right ) \ed \left (\mathcal{R}^g(t+t')\mid\mathcal{R}(t)=(\rho, \mathbf{0}_\infty)\right ), \]
  which is now an immediate consequence of \textit{iii)} in Definition \ref{snec}.
\end{proof}

Let us now investigate the transition rates of the Markov chains $\mathcal{R}_{| {n} }$ appearing in Lemma \ref{lem:projMarkov}, for every $n\in\mathbb{N}$.
In this direction fix $n\in\N$, let $\rho\in\cN_\infty$ and $\pi\in\cN_{n}$ and denote the jump rate of $\mathcal{R}_{|n}$ from $\rho_{|n}$ to $\pi$ by
\begin{equation}\label{rate}
q_n({\rho,{\pi}}) := \lim_{t\rightarrow 0+} \frac{1}{t} \mathbb{P}_\rho( \mathcal{R}_{|n}(t) = {\pi})
\end{equation}
where $\P_\rho(\cdot)=\P(\cdot\,|\, \mathcal{R}(0)= \rho )$.
The index $n$ is not necessary in the notation as it can be read in the partition $\pi$.
However we keep it as it will ease reading. Remind that $q_n({\rho,{\pi}})$ only depends on $\rho$ through $\rho_{|n}$.
As is remarked in Lemma \ref{lem:projMarkov}, $q_n({\rho,\pi})$ equals zero if $\pi$ is not obtained from $\rho_{|n}$ by coagulating blocks according to a partition in $\widetilde{\mathcal{P}}(\rho_{|n})$, that is by merging
some species blocks of $\rho^s_{|n}$ into one and some gene blocks of the new species into one. 
Also observe that the rates do not depend on the sizes of the gene blocks in the starting configuration so there is no loss of generality if we consider that $\rho^g=\mathbf{0}_\infty$, the trivial partition made of singletons.
Of course changing the starting partition $\rho$ has some effect on the arrival partition $\pi$.
This is why we will need to write transition rates in another way, giving more emphasis on the dependence of the coagulation mechanism upon the starting partition.

Fix $n\in\bar\N$ and suppose that $\mathcal{R}_{|n}$ starts from $n$ singleton gene blocks allocated into $b$ species.
Since labels of genes do not affect the transition rates, we will keep the data of the number of genes in each species in a vector $\mathbf{g}=(g_1,\dots, g_b)$. 
This vector suffices to describe the starting position. 
Indeed $|\mathbf{g}|=b$ gives the number of species and $\sum_{i=1}^bg_i=n$ gives the number of genes.

Now the coagulation mechanism will be described by two terms. We will say that a gene block \emph{participates in the coalescence event} if it merges with other gene blocks. We will say that a species block \emph{participates in the coalescence event} if it merges with other species blocks or if it contains gene blocks that participate in the coalescence event. 

The behaviour of the species blocks will be encoded in a vector $\mathbf{s}=(s_1,\dots, s_b)$ with coordinates taking values in $\{0,1\}$. Namely, $s_i=1$ if the $i$-th species participates in the coalescence event and $s_i=0$ otherwise.
The total number of species involved in the event is $k=\sum_{i=1}^b s_i$.

The behaviour of the gene blocks will be encoded by an array $\mathbf{c}=(\mathbf{c}_1,\dots, \mathbf{c}_b)$ where $\mathbf{c}_i$ is a vector describing which gene blocks in the $i$-th species participate in the coalescence event. 
If $s_i=1$ ($i$-th species block participating in the coalescence event), then $\mathbf{c}_i=(c_{i1},\dots,  c_{ig_i})$ is such that $c_{ij}=1$ if $j$-th gene block inside $i$-th species block participates in the coalescence event and $c_{ij}=0$ otherwise.
If $s_i=0$, the $i$-th species block is not participating in the event and so neither will the gene blocks within it. 
In this case we set $\mathbf{c}_i=(0,0,\ldots,0)=\mathbf{0}$ and nothing happens at the gene level.
Note that the number of gene blocks participating in the coalescence event is $\sum_{i,j}c_{ij}$.

Note that all such arrays $ (\mathbf{g}, \mathbf{s}, \mathbf{c}) $ do not necessarily code for observable coalescence events, so we will define a restricted set of arrays of interest for our study.
First, note that one needs to have $\sum_i s_i \geq 2$ in order to observe a species merger.
If $\sum_i s_i = 1$, then there is a gene coalescence if and only if $\sum_{i,j} c_{ij} \geq 2$.
Also, we will restrict ourselves to the arrays $(\mathbf{g}, \mathbf{s}, \mathbf{c})$ such that $\sum_{i,j} c_{ij} \neq 1$, because \textit{a sole gene coalescing} is not distinguishable from \textit{no gene coalescing}.

Formally, we consider finite arrays $ (\mathbf{g}, \mathbf{s}, \mathbf{c}) $ satisfying the assumptions
\begin{align}
&\begin{aligned}
&\text{If } |\mathbf{g}|=b, \text{ then } \mathbf{s}\in\{0,1\}^{b} \text{ and } \mathbf{c}=(\mathbf{c}_1, \ldots, \mathbf{c}_b), \; \mathbf{c}_i\in\{0,1\}^{g_i}, \\ & \text{ with } s_i=0 \implies \forall j, c_{ij}=0
\end{aligned} \tag{\textbf{H1}} \label{eq:hypothesesarray1} \\
&  {\sum_{i,j} c_{ij} < 2 \implies \sum_i s_i \geq 2, \tag{\textbf{H2}}} \label{eq:hypothesesarray2} \\
\text{and }& \sum_{i,j} c_{ij} \neq 1. \tag{\textbf{H3}}\label{eq:hypothesesarray3}
\end{align}
We denote by $\mathcal{C}$ the set of arrays $ (\mathbf{g}, \mathbf{s}, \mathbf{c}) $ satisfying (\textbf{H1}), (\textbf{H2}) and (\textbf{H3}).

We then denote the transition rate of $\mathcal{R}_{|n}$ from a partition described by $\mathbf{g}$ (such that $\sum g_i=n$) to a new partition obtained by merging species and genes according to $\mathbf{s}$ and $\mathbf{c}$ by
\[
q_{b,k}(\mathbf{g}, \mathbf{s}, \mathbf{c}).
\]
Here again indices $b$ and $ k$ are not necessary but permit to read easily the coalescence event at the species level ($k=\sum s_i$ species merging among $b=|\mathbf{g}|$).
We insist on the fact that we consider only arrays $(\mathbf{g}, \mathbf{s}, \mathbf{c}) \in \mathcal{C}$ when we study the rates $q_{b,k}(\mathbf{g}, \mathbf{s}, \mathbf{c})$, and that these quantities determine uniquely the law of a SNEC $\mathcal{R}$, since they describe completely the rates associated to each finite-space continuous-time Markov chain $\mathcal{R}_{|n}$.

We introduce a notation that we will use in the next result for ease of writing.
For $\mu$ a probability on $[0,1]$, consider any probability space where $Z_1, Z_2, \ldots$ are i.i.d.\ with distribution $\mu$ and denote the expectation $\E_{\mu}$.
Now take a vector $(g_i,\, i\in S)$ of integers, where $S$ is a finite subset of $\N$.
We define
\begin{equation}\label{eq:deftruc}
\begin{aligned}
\mathcal{U}(\mu, (g_i, i\in S)) 
&= \E_{\mu}\Big [\sum_{i\in S}g_i Z_i(1-Z_i)^{g_i-1} \!\!\!\prod_{j\in S\,:\, j\neq i}(1-Z_j)^{g_j}\Big ]\\
&=  \sum_{i\in S} g_{i}\int_{[0,1]}\mu(dq)q(1-q)^{g_{i}-1} \!\!\!\prod_{j\in S\,:\, j\neq i}\int_{[0,1]}\mu(dq)(1-q)^{g_{j}}.
\end{aligned}
\end{equation}
This can be thought of as the probability that a random array $(c_{ij}, \,i\in S, 1\leq j \leq g_i)$ does not satisfy \eqref{eq:hypothesesarray3}, where conditional on $(Z_i,\, i \in S)$ the variables $(c_{ij})$ are independent, and for all $i,j$, $c_{ij}=1$ with probability $Z_i$. 
We can now state our main result.

\begin{theorem}\label{thm:charac}
  There exist two non-negative real numbers $a_s, a_g\geq 0$ and two measures:
  \begin{itemize}
    \item $\nu_s$ on $E=(0,1]\times\mathcal M_1([0,1])$;
    \item $\nu_g$ on $(0,1]$;
  \end{itemize}
  such that
  \begin{subequations}\label{eq:condspecies}
    \begin{gather}
    \int_E\nu_s(dp,d\mu)p^2<\infty, \label{eq:condspecies1} \\
    \int_E\nu_s(dp,d\mu)p\int_{[0,1]}\mu(dq)q^2<\infty, \label{eq:condspecies2}
    \end{gather}
  \end{subequations}
  \begin{equation}
  \text{ and }\int_{(0,1]}\nu_g(dq)q^2 < \infty, \label{eq:condgenes}
  \end{equation}
  and such that for any array $(\mathbf{g}, \mathbf{s}, \mathbf{c})\in\mathcal{C}$ such that 
  $|\mathbf{g}|=b$, $\sum_i s_i=k$ and $\sum_j c_{ij}=l_i$,
  \begin{equation}\label{eq:2species}
  \begin{aligned}
  q_{b,k}(\mathbf{g}, \mathbf{s}, \mathbf{c}) = &
  \int_E \nu_s(dp, d\mu)p^{k}(1-p)^{b-k} \Biggl (\prod_{i\,:\,{s_i}=1}\int_{[0,1]}\mu(dq)q^{l_i}(1-q)^{g_{i}-l_i} \\
  & \qquad \qquad \qquad \qquad + \1{\{\mathbf{c}=\mathbf{0}\}}\, \mathcal{U}(\mu, (g_i, \,1\leq i \leq b \text{ with } s_i = 1))\Biggr ) \\
  &+ a_s \1{\left \{k = 2, \mathbf{c}=\mathbf{0}\right \}}\\
  &+ \1{\{k=1\}}\left (a_g \1{\{l_I = 2\}} + \int_{(0,1]}\nu_g(dq)q^{l_I}(1-q)^{g_I-l_I}  \right ),
  \end{aligned}
  \end{equation}
  where the functional $\mathcal{U}$ is defined in \eqref{eq:deftruc} and $I=I(\mathbf{g}, \mathbf{s}, \mathbf{c})$, in the case $k=1$, is the unique index in $\{1, 2,\ldots,b\}$ such that $s_I=1$.
  
  {Furthermore, this correspondence between laws of SNEC processes and quadruplets $(a_s,a_g,\nu_s,\nu_g)$ satisfying \eqref{eq:condspecies} and \eqref{eq:condgenes} is bijective.}
\end{theorem}

\begin{remark}{
    We will show the \emph{surjective} part of the theorem's last statement in Section \ref{sec:poisson}, using an explicit Poissonian construction.
    For now we prove the existence and uniqueness of the characteristics $(a_s,a_g,\nu_s,\nu_g)$.}
\end{remark}

\section{Proof of Theorem \ref{thm:charac}} \label{sec:proof}
Consider a SNEC process $\mathcal R=((\mathcal R^s(t),\mathcal R^g(t)),\; t\geq0)$ with values in $\mathcal N_\infty$ and recall its jump rates $q_n(\rho,\pi)$ defined in \eqref{rate}.
Also recall the alternative notation $q_{b,k}(\mathbf{g},\mathbf{s},\mathbf{c})$.
Here, $\mathbf{g}$ is a vector of size $b$ such that $\sum g_i=n$,
$\mathbf{s}$ is a vector having  the same size as $\mathbf{g}$ with coordinates in $\{0,1\}$ such that $\sum s_i=k$, and
$\mathbf{c}$ is a family of $|\mathbf{g}|$ elements denoted by $\mathbf{c}_1,\mathbf{c}_2\dots$ where $\mathbf{c}_i$ is a vector of $\{0,1\}^{g_i}$ if $s_i=1$ and $\mathbf{c}_i=\mathbf{0}$ if $s_i=0$.

\begin{lemma}
  For any initial value $\rho = (\rho_s,\rho_g)\in \mathcal{N}_{\infty}$, there exists a unique measure $\mu_{\rho}$ on $\mathcal{N}_{\infty}$ such that
  \begin{equation}\label{eq:mu_conditions}
  \mu_{\rho}(\{\rho\}) = 0 \quad \text{and} \quad \forall n\geq 1, \; \mu_{\rho}(\Pi_{|n} \neq \rho_{|n}) < \infty
  \end{equation}
  and such that the transition rate of the Markov chain $\mathcal{R}_{|n}$ from $\rho_{|n}$ to $\pi\in \mathcal{N}_{n}$ is given by
  \begin{equation}\label{eq:rates_from_mu}
  q_n(\rho,\pi) = \mu_{\rho}(\Pi_{|n} = \pi).
  \end{equation}
  Furthermore, for any permutation $\sigma:\N\to\N$,
  \begin{equation}\label{eq:exch_mu_rho}
  \mu_{\rho}(\sigma(\Pi)\in \cdot) = \mu_{\sigma(\rho)}(\Pi\in \cdot).
  \end{equation}
\end{lemma}
Note that we write $\mu_{\rho}(\Pi \in A)$ instead of $\mu_{\rho}(A)$ because we implicitly work on the canonical space  $\mathcal{N}_\infty$ and we denote by $\Pi$ the generic element of $\mathcal{N}_\infty$.

\begin{proof}
  Let $n < m$. We first note that since $\mathcal{R}_{|m}$ and $\mathcal{R}_{|n} = (\mathcal{R}_{|m})_{|n}$ are Markov chains, the transition rates can be expressed, for any $\pi\in\mathcal{N}_n\setminus\{\rho_{|n}\}$,
  \begin{equation}\label{eq:rates_additivity}
  q_n(\rho,\pi) = \sum_{\pi' \in\mathcal{N}_{m}\,:\;\pi'_{|n} = \pi} q_m(\rho,\pi'). 
  \end{equation}
  Let us now check that this consistency property along with Carathéodory's extension theorem ensures us that there exists a measure $\mu_{\rho}$ on $\mathcal{N}_\infty\setminus\{\rho\}$ satisfying \eqref{eq:rates_from_mu}.
  
  Here the family $\mathcal{A} := \big\{\{\Pi_{|n} = \pi\} ,\; n\in\N, \pi\in \mathcal{N}_n\setminus \{\rho_{|n}\}\big\}\cup\{\emptyset\}$ clearly forms a semi-ring of subsets of $\mathcal{N}_\infty$, and it remains to check that the functional $\widetilde{\mu}:\mathcal{A} \to [0,+\infty]$, defined by
  \[ \widetilde{\mu}(\emptyset) := 0 \quad \text{and} \quad \widetilde{\mu}(\{\Pi_{|n} = \pi\}) := q_n(\rho, \pi), \]
  is a pre-measure.
  Equation \eqref{eq:rates_additivity} shows that $\widetilde{\mu}$ is finitely additive, and the only difficulty lies in understanding that $\widetilde{\mu}$ is countably additive.
  Now observe that the topology of $\mathcal{N}_\infty\setminus\{\rho\}$ is generated by $\mathcal{A}$, and that each of the non-empty sets in $\mathcal{A}$ is both open and closed (thus compact), because
  \[ \mathcal{N}_\infty \setminus \{\Pi_{|n} = \pi\} = \bigcup_{\rho\in\mathcal{N}_n\setminus\{\pi\}}  \{\Pi_{|n} = \rho\}. \]
  This implies that if $(A_n)_{n\geq 1}$ is a family of pairwise disjoint elements of $\mathcal{A}$ such that $\bigcup_{n}A_n\in\mathcal{A}$, then at most a finite number of the $A_n$ are non-empty (because since $\bigcup_{n}A_n$ is compact, there is a finite subcover), so countable additivity reduces to finite additivity.
  Therefore Carathéodory's extension theorem applies, hence the existence of a measure $\mu_{\rho}$ on $\mathcal{N}_\infty\setminus\{\rho\}$ satisfying \eqref{eq:rates_from_mu}.
  
  Considering $\mu_{\rho}$ as a measure on $\mathcal{N}_\infty$ such that $\mu_{\rho}(\{\rho\})=0$, we check easily \eqref{eq:mu_conditions} by noticing that 
  \[ \mu_{\rho}(\Pi_{|n} \neq \rho_{|n}) = \sum_{\pi\in\mathcal{N}_n\setminus\{\rho_{|n}\}} q_n(\rho,\pi) < \infty. \]
  Furthermore, for any $n$, $\pi\in\mathcal{N}_n\setminus\{\rho_{|n}\}$ and $\sigma:\N\to\N$ permutation, we have by the exchangeability property \eqref{eq:defSNECexch} of a SNEC, that
  \[ \mu_{\rho}(\sigma(\Pi)_{|n} =\pi) = \lim_{t\to 0}\frac{1}{t} \P_{\rho}(\sigma(\mathcal{R}(t))_{|n} = \pi) = \lim_{t\to 0}\frac{1}{t} \P_{\sigma(\rho)}(\mathcal{R}_{|n}(t) = \pi) = \mu_{\sigma(\rho)}(\Pi_{|n} =\pi), \]
  which proves that \eqref{eq:exch_mu_rho} holds on $\mathcal A$. Since the topology of $\mathcal{N}_\infty$ is generated by $\mathcal A$, the proof is complete.
\end{proof}

The latter lemma implies that there exists a family of exchangeable measures on $\cN_\infty$ characterizing (i.e.\ acting as an analog of a Markov kernel for continuous-space pure-jump Markov chains) the SNEC process $\mathcal{R}$.
Furthermore, since we are dealing with a simple coalescent, it is clear from the characterization \eqref{eq:rates_from_mu} that $\mu_{\rho}$ is simple in the sense that it is supported by all the possible bivariate partitions obtained from a simple coalescence from $\rho$. To put it simply,
\[ \mu_\rho\left (\cN_{\infty}\setminus\{\text{Coag}_{2}(\rho, \tilde\pi), \; \tilde\pi\in\widetilde{\mathcal{P}}(\rho)\}\right ) = 0. \]
The measure $\mu_\rho$ can be translated as a measure on arrays of random variables in $\{0,1\}$.
Informally, we can associate to each species in $\rho$ a 1 entry if it participates in the coalescence and a 0 entry otherwise. Inside the species participating to the coalescence event, we can also associate a 1 entry to the genes participating in the coalescence event and a 0 entry otherwise.
To tally with the definition of the SNEC we will need a certain partial exchangeability structure for this array.  
This picture can be formalized as follows.
Let $((X_i, (Y_{ij},\;j\in\mathbb N)),\;i\in\mathbb N)$ be an array of Bernoulli random variables and
denote by $Z_i$ the $i$-th line vector $(X_i, (Y_{ij},\;j\in\mathbb N))$.
We say that this array is \emph{hierarchically exchangeable} if
\begin{enumerate}[label=\textbf{(A\arabic*)}]
  \item the family $(Z_i,\;i\in\mathbb N)$ is exchangeable\label{A1};
  \item \label{A2} for any $i\in\mathbb N$, the family $(X_i,(Y_{ij},\;j\in \mathbb N))$ is invariant under any permutation over the $j$'s.
\end{enumerate}

We also naturally extend this definition to measures on the space $(\{0,1\}\times\{0,1\}^{\N})^{\N}$.
We say that such a measure $\mu$ is \textit{hierarchically exchangeable} if it is invariant both under the permutations of the rows, and the permutations within a row.

For an initial state $\rho=(\rho^{s}, \rho^{g})\in\mathcal{N}_\infty$ and an arrival state $\pi=\text{Coag}_{2}(\rho, \tilde\pi)\in\mathcal{N}_\infty$, with $\tilde\pi$ a simple bivariate partition $\tilde\pi=(\tilde\pi^{s}, \tilde\pi^{g})\in(\mathcal{P}'_{\infty})^{2}$, define the array $\mathbf{Z}(\rho,\pi)=(\mathbf{X}, \mathbf{Y}_1,\mathbf{Y}_2,\ldots)$ by
\begin{equation}\label{eq:def_Z}
\begin{aligned}
&X_i = 1 \text{ if the $i$-th block in $\rho^{s}$ has coalesced in $\pi$},\\
&Y_{ij} = 1 \text{ if the $I(i,j)$-th block in $\rho^{g}$ has coalesced in $\pi$},
\end{aligned}
\end{equation}
where $I(i,j):=k$ if the $k$-th block of $\rho^{g}$ is the $j$-th gene block of the $i$-th species block.

Now choose a state $\rho$ with an infinite number of species blocks, each containing an infinite number of gene blocks.
Let $\nu$ be the push-forward of $\mu_\rho$ by the application
\[ \pi \longmapsto \mathbf{Z}(\rho,\pi). \]
Then the exchangeability property of $\mu_{\rho}$ \eqref{eq:exch_mu_rho} implies that $\nu$ is a hierarchically exchangeable measure on $(\{0,1\}\times\{0,1\}^{\N})^{\N}$, and \eqref{eq:mu_conditions} implies that 
\begin{equation}\label{eq:nu_conditions}
\nu(\mathbf{Z}=\mathbf{0}) = 0,\quad\text{ and }\quad \nu\left (\sum_{i=1}^{n}X_i\geq 2 \text{ or } \exists i \leq n, \, \sum_{j=1}^{n}Y_{ij}\geq 2\right ) <\infty,
\end{equation}
where $\mathbf{0}$ denotes the null array on $(\{0,1\}\times\{0,1\}^{\N})^{\N}$.
Also, note that the application $(\mu_\rho, \,\rho\in \mathcal{N}_\infty) \mapsto \nu$ is one-to-one.
Indeed, we can conversely define for any $\mathbf{Z}$ and any nested partition $\rho \in\mathcal{N}_\infty$, the nested partition $C(\rho, \mathbf{Z}) \in \mathcal{N}_\infty$ obtained from $\rho$ by merging exactly the blocks that \emph{participate in the coalescence} where 
\begin{itemize}
  \item The $i$-th block of $\rho^s$ participates iff $X_i=1$;
  \item The $j$-th block in $\rho^g$ of the $i$-th block of $\rho^s$ participates iff $X_i=1$ and $Y_{ij}=1$.   
\end{itemize}
With this definition, $\mu_\rho$ is obtained as the push-forward of $\nu$ by the map $\mathbf{Z} \mapsto C(\rho,\mathbf{Z})$.

Now recall the alternative notation $q_{b,k}(\mathbf{g},\mathbf{s},\mathbf{c})$ for the transition rate of $\mathcal{R}_{|n}$ (where $n=\sum_ig_i$) from a nested partition with $b$ species blocks and $g_1,\dots,g_b$ gene blocks inside them, to a nested partition obtained by merging $k$ species blocks according to the vector $\mathbf{s}$ and gene blocks inside those species according to the array $\mathbf{c}$.
For any array $(\mathbf{g},\mathbf{s},\mathbf{c})\in\mathcal{C}$, note that \eqref{eq:rates_from_mu} translates in terms of our push-forward $\nu$ in the following way:
\begin{equation}\label{eq:rates_from_nu}
\begin{aligned}
q_{b,k}(\mathbf{g},\mathbf{s},\mathbf{c})& = \nu\bigl (\forall 1\leq i\leq b,\,  X_i = s_i,\text{ and }\forall 1\leq j \leq g_i, Y_{ij} = c_{ij}\bigr ) \\ 
&+ \1_{\{\mathbf{c}=\mathbf{0}\}} \nu \left (\forall 1\leq i\leq b,\,  X_i = s_i,\text{ and }\sum_{i=1}^{b} \sum_{j=1}^{g_i} Y_{ij} = 1 \right ).
\end{aligned}
\end{equation}
Indeed, the first line is quite straightforward and comes from our representation of coalescence events by those arrays $(\mathbf{g},\mathbf{s},\mathbf{c})\in\mathcal{C}$ (see Section \ref{sec:snec}) which basically means that blocks participating in a coalescence event are those associated with a $1$.
However in the case when $\mathbf{c}=\mathbf{0}$, there is an additional probability to observe the coalescence of species blocks associated to $\mathbf{s}$ with no coalescence of gene blocks (the case when all the $Y_{ij}$'s are 0 is included in the first term), which is when exactly one of the $Y_{ij}$'s is equal to $1$.
This gives rise to the second line of \eqref{eq:rates_from_nu}.

We now have to establish a de Finetti representation of hierarchically exchangeable arrays to express the measure of an event of the form $\{\forall 1\leq i\leq b,\,  X_i = s_i,\text{ and }\forall 1\leq j \leq g_i, Y_{ij} = c_{ij}\}$.
Note that we consider random measures in the following, but only on Borel spaces $(S,\mathcal{S})$ (i.e.\ spaces isomorphic to a Borel subset of $\R$ endowed with the Borel $\sigma$-algebra), which will enable us to use de Finetti's theorem \cite{kallenbergProbabilistic2005}.
For this we write $\mathcal{M}_1(S)$ for the space of probability measures on $S$, which is endowed with the $\sigma$-algebra generated by the maps $\mu \mapsto \mu(B)$ for all $B\in\mathcal{S}$.
The spaces $(S,\mathcal{S})$ that we consider will be for instance $[0,1]$ with its Borel sets or $\{0,1\}^{\N}$ equipped with the product $\sigma$-algebra, which are clearly Borel spaces.

\begin{proposition}\label{prop:deFinetti}
  Let $\mathbf{Z}=(Z_i,\;i\in\mathbb N)=((X_i, (Y_{ij},\;j\in\mathbb N)),\;i\in\mathbb N)$ be a hierarchically exchangeable array (with variables in $\{0,1\}$). 
  Then there exists a unique probability measure $\Lambda$ on $E'=[0,1]\times\mathcal M_1([0,1])\times\mathcal M_1([0,1])$ (and we will write any element $\mu$ of $E'$ as $(p,\mu_0,\mu_1)$) such that for all $n\geq 1$
  \begin{align}\label{eq:deFinetti}
  &\P(X_i=x_i, Y_{ij}=y_{ij},\;i,j\in[n])\nonumber\\
  &\begin{aligned}
  &=\int_{E'} \Lambda(d\mu)\prod_{i=1}^n\Bigg[(p\1_{\{x_i=1\}}+(1-p)\1_{\{x_i=0\}}) \\
  & \hspace{10em} \int_{[0,1]}\mu_{x_i}(dq_i)\prod_{j=1}^n(q_i\1_{\{y_{ij}=1\}}+(1-q_i)\1_{\{y_{ij}=0\}})\Bigg].
  \end{aligned}
  \end{align}
\end{proposition}
\begin{proof}
  Let us first observe that if a sequence $(X,(Y_{j},\;j\in \mathbb N))$ satisfies Hypothesis \ref{A2}, then, conditional on $X=x\in\{0,1\}$, the sequence $(Y_{j},\;j\in \mathbb N)$ is exchangeable.
  We can thus apply de Finetti's theorem:
  conditional on $X=x$ there is a unique probability measure $\mu_x$ giving the distribution of the asymptotic frequency $q$ of the variables ${(Y_{j},\;j\in \mathbb N)}$, and conditional on $q$ they are i.i.d. Bernoulli with parameter $q$. 
  This implies that, for any $\{0,1\}$-valued finite sequence $(x, y_1, y_2,\dots, y_k)$,
  \begin{align}\label{eq:cddeF}
  &\P(X=x, Y_1=y_1, \dots, Y_k = y_k) \nonumber \\
  &=\P(X=x)\int_{[0,1]}\mu_{x}(dq)\prod_{j=1}^{k}(q\1_{\{y_{j}=1\}}+(1-q)\1_{\{y_j=0\}}).
  \end{align}
  Also observe that since $X$ is binary, there exists $p\in[0,1]$ such that $\P(X=x)=p\1_{\{x=1\}}+(1-p)\1_{\{x=0\}}$.
  
  As a consequence of Hypothesis \ref{A1}, we can apply once again de Finetti's theorem: 
  there exists a unique law $\tilde\Lambda$ on $\mathcal{M}_1(\{0,1\}^{\mathbb N})$ such that the law of $(Z_i,\;i\in\N)$ equals 
  $\int_{\mathcal{M}_1(\{0,1\}^{\mathbb N})}\tilde\Lambda(d\tilde\mu)\bigotimes_{i\geq1}\tilde\mu$.
  Furthermore it has been seen that $\tilde\mu$ can be expressed as in \eqref{eq:cddeF}.
  
  Now let $F$ stand  for the measurable mapping such that
  $F(\tilde\mu)=(p,\mu_0,\mu_1)\in E'$ and let $\Lambda$ be the push-forward of $\tilde\Lambda$ by the mapping $F$. 
  We obtain that if $A$ and $(B_i,\;i\in A)$ are finite subsets of $\N$, then
  \begin{align*}
  &\P(X_i=x_i, Y_{ij_i}=y_{ij_i},\;i\in A,j_i\in B_i)\\
  &=\int_{E'} \Lambda(d\mu)\prod_{i\in A}\Bigg[(p\1_{\{x_i=1\}}+(1-p)\1_{\{x_i=0\}}) \\ 
  & \hspace{10em} \int_{[0,1]}\mu_{x_i}(dq_i)\prod_{j_i\in B_i}(q_i\1_{\{y_{ij_i}=1\}}+(1-q_i)\1_{\{y_{ij_i}=0\}})\Bigg].
  \end{align*}
  This ends the proof.
\end{proof}

This result is almost enough to express \eqref{eq:rates_from_nu} but one has to be careful because the measure $\nu$ might not be finite.
However, it is $\sigma$-finite because by \eqref{eq:nu_conditions},
\[ 
\nu = \lim_{n\to\infty} \uparrow \nu\left (\left \{ \sum_{i=1}^{n}X_i\geq 2 \text{ or } \exists i \leq n, \, \sum_{j=1}^{n}Y_{ij}\geq 2 \right \}\cap \; \cdot\;\right ),
\]
and those events have finite measure.
The idea behind the following lemma is to make use of those events and hierarchical exchangeability to express $\nu$ as a limit of finite measures which, thanks to an application of Proposition \ref{prop:deFinetti}, have a representation under the form $\eqref{eq:deFinetti}$.

Let us introduce some notation that will enable us to make this argument formal.
For a fixed vector $(\mathbf{g},\mathbf{s},\mathbf{c})\in \mathcal{C}$, such that $|\mathbf{g}| = b$, let us examine the event 
\[ A=A(\mathbf{g},\mathbf{s},\mathbf{c}):=\{\forall 1\leq i\leq b,\,  X_i = s_i,\text{ and }\forall 1\leq j \leq g_i, Y_{ij} = c_{ij}\} \]
and its measure $\nu(A)$.
Let us define, for all $n\geq 1$ the shifted random array
\begin{equation}\label{eq:shift_array}
\mathbf{Z}_n := (X_{i+n}, Y_{(i+n)j},\;i,j \in\N).
\end{equation}
We decompose naturally $A = (A\cap\{\mathbf{Z}_b \neq\mathbf{0}\}) \cup (A\cap\{\mathbf{Z}_b =\mathbf{0}\})$, where $b=|\mathbf{g}|$.

{Recall that the array $\mathbf{Z}$ encodes which species blocks and which gene blocks are participating in a coalescence.
  Therefore the event $A\cap\{\mathbf{Z}_b \neq 0\}$ indicates that there are merging species blocks outside of the first $b$ blocks.
  In fact we will see that this implies that such merging blocks are infinitely many (a random proportion $p$ of them), and within each of these blocks, a random proportion $q$ of gene blocks are also participating in the coalescence event.
  The following technical lemma makes this statement rigorous.}

\begin{lemma}\label{lem:formulaZneq0}
  For an array $(\mathbf{g},\mathbf{s},\mathbf{c})$ satisfying assumptions \eqref{eq:hypothesesarray1} and \eqref{eq:hypothesesarray2}, there exists a unique measure $\nu_s$ on $E=(0,1]\times\mathcal{M}([0,1])$ such that
  \begin{equation}\label{eq:formulaZneq0}
  \nu(A\cap\{\mathbf{Z}_b\neq \mathbf{0}\}) = \int_{E} \nu_s(dp, d\mu)p^k(1-p)^{b-k}\prod_{i\,:\,s_i=1}\int_{[0,1]}\mu(dq) q^{l_i}(1-q)^{g_i-l_i},
  \end{equation}
  where $b:=|\mathbf{g}|$, $k:=\sum_i s_i$ and $l_i := \sum_j c_{ij}$.
  Moreover, $\nu_s$ satisfies \eqref{eq:condspecies}.
\end{lemma}

\begin{proof}
  We define some events that will be used to express $\nu(A\cap\{\mathbf{Z}_b \neq\mathbf{0}\})$.
  \begin{gather*}
  A_n=A_n(\mathbf{g},\mathbf{s},\mathbf{c}):=\{\forall 1\leq i\leq b,\,  X_{i+n} = s_i,\text{ and }\forall 1\leq j \leq g_i, Y_{(i+n)j} = c_{ij}\} \\
  {B_n := \left\{ \sum_{i=1}^{n}X_i \geq 2 \right\}}\\
  {B'_n = B'_n(\mathbf{g},\mathbf{s},\mathbf{c}):= \left\{ \sum_{i=b+1}^{b+n}X_i \geq 2 \right\}.}
  \end{gather*}
  {Note that $(\mathbf{g},\mathbf{s},\mathbf{c})$ satisfies \eqref{eq:hypothesesarray1} and \eqref{eq:hypothesesarray2}, so we have ${A\subset \{\sum_{i=1}^{m}X_i \geq 2 \text{ or }\sum_{i,j=1}^{m}Y_{ij} \geq 2\}}$ for $m=\max(b,g_1,\ldots g_b)$.
    Now because $\nu$ satisfies \eqref{eq:nu_conditions}, necessarily $\nu(A)<\infty$, which implies that
    \[
    \nu(A\cap\{\mathbf{Z}_b \in \cdot\})
    \]
    is a finite hierarchically exchangeable measure on $(\{0,1\}\times\{0,1\}^{\N})^{\N}$.
    The de Finetti representation (Proposition \ref{prop:deFinetti}) implies that on the event $A$, $\mathbf{Z}_b$ is either $\mathbf{0}$, or has an infinite number of entries with value $1$.
    In particular, 
    \[ A\cap\{\mathbf{Z}_b \neq \mathbf{0} \} = A\cap\{\mathbf{Z}_b \text{ has at least two entries at }1\}, \]
    therefore, there is the equality
    \[
    A\cap\{\mathbf{Z}_b \neq\mathbf{0}\} = \bigcup_{n\geq 1} A\cap B'_n,
    \]
    where the union is increasing.}
  Therefore,
  \begin{align*}
  \nu(A\cap\{\mathbf{Z}_b \neq\mathbf{0}\}) &= \lim_{n\to\infty}\nu(A\cap B'_n). \\
  &= \lim_{n\to\infty} \nu(B_n \cap A_n),
  \end{align*}
  where we used the hierarchical exchangeability of $\nu$ to get the second equality.
  Now we know from \eqref{eq:nu_conditions} and because $\nu$ is exchangeable that the measure
  \[ \nu(B_n \cap \{\mathbf{Z}_n \in \cdot \}) \]
  is a finite hierarchically exchangeable measure on $(\{0,1\}\times\{0,1\}^{\N})^{\N}$.
  Because it is finite we can apply Proposition \ref{prop:deFinetti} to deduce that there exists a finite measure $\Lambda_n$ on $E' = (0,1]\times\mathcal{M}([0,1])^2$ such that
  \begin{align*}
  \nu(B_n \cap A_n)& = \int_{E'} \Lambda_n(dp, d\mu_0,d\mu_1)\prod_{i=1}^{b}\Bigg[(p\1_{\{s_i=1\}}+(1-p)\1_{\{s_i=0\}}) \\
  &\hspace{12.5em} \int_{[0,1]}\mu_{s_i}(dq_i)\prod_{j=1}^{g_i}(q_i \1_{\{c_{ij}=1\}}+(1-q_i)\1_{\{c_{ij}=0\}})\Bigg].
  \end{align*}
  We can simplify this expression since $\nu$ is supported by the set $\{\forall i \in \N, \, X_i = 0 \Rightarrow \forall j\in \N, Y_{ij}=0 \}.$
  This implies that $\Lambda_n$-a.e.\ the measure $\mu_0$ is $\delta_0$ the Dirac measure at $0$.
  Therefore we write $\widetilde{\Lambda}_n$ for the push forward measure on $E:=(0,1]\times\mathcal{M}([0,1])$ of $\Lambda_n$ by the application $(p,\mu_0,\mu_1)\mapsto (p,\mu_1)$.
  We now have
  \begin{equation}\label{eq:p_et_mu}
  \nu(B_n \cap A_n) = \int_{E} \widetilde{\Lambda}_n(dp, d\mu)p^k(1-p)^{b-k}\prod_{i\,:\,s_i=1}\int_{[0,1]}\mu(dq) q^{l_i}(1-q)^{g_i-l_i} . 
  \end{equation}
  To be able to pass to the limit, let us check that the sequence of measures $(\widetilde{\Lambda}_n)$ is increasing.
  Indeed, recall that $\Lambda_n$ is obtained from two applications of de Finetti's theorem to the exchangeable array $\mathbf{Z}_n$, so the asymptotic parameters $p$ and $\mu$ appearing in \eqref{eq:p_et_mu} are a deterministic, measurable functional of $\mathbf{Z}_n$.
  Let us write this functional $F(\mathbf{Z}_n)=(p,\mu)$, so now $\widetilde{\Lambda}_n$ is simply the measure
  \[ \nu(B_n\cap \{F(\mathbf{Z}_n)\in \cdot \}). \]
  But $p$ and $\mu$ are asymptotic quantities of the array $\mathbf{Z}_n$, which do not depend on the first row of $\mathbf{Z}_n$, so $F(\mathbf{Z}_{n+1})=F(\mathbf{Z}_{n})$ and we have
  \begin{align*}
  \widetilde{\Lambda}_n &= \nu(B_n\cap \{F(\mathbf{Z}_n)\in \cdot \})\\
  &= \nu(B_n\cap \{F(\mathbf{Z}_{n+1})\in \cdot \})\\
  &\leq \nu(B_{n+1}\cap \{F(\mathbf{Z}_{n+1})\in \cdot \})\\
  &= \widetilde{\Lambda}_{n+1},
  \end{align*}
  where the passage from the second to the third line is simply because $B_n\subset B_{n+1}$.
  Therefore there is a limiting measure $\nu_{s}$ on $E$ such that
  \begin{equation*}
  \nu(A\cap\{\mathbf{Z}_b\neq \mathbf{0}\}) = \lim_{n\to\infty}\nu(B_n \cap A_n) = \int_{E} \nu_s(dp, d\mu)p^k(1-p)^{b-k} \!\!\! \prod_{i\,:\,s_i=1}\int_{[0,1]}\mu(dq) q^{l_i}(1-q)^{g_i-l_i},
  \end{equation*}
  so we recover \eqref{eq:formulaZneq0}.
  {To prove the uniqueness of this measure, consider any measure $\nu'_s$ on $E$ such that \eqref{eq:formulaZneq0} holds.
    Then we have simply
    \[
    \tilde\Lambda_n(dp,d\mu) = \nu(B_n\cap \{F(\mathbf{Z}_n)\in (dp,d\mu) \}) = (1-(1-p)^n-np(1-p)^{n-1})\nu'_s(dp,d\mu),
    \]
    where the first equality is by definition and the second because we assumed that \eqref{eq:formulaZneq0} holds for $\nu'_s$.
    Taking limits on both sides yields
    \[
    \nu_s(dp,d\mu) = \nu'_s(dp,d\mu).
    \] }
  It remains to prove \eqref{eq:condspecies}.
  Note that the condition \eqref{eq:nu_conditions} implies that
  \begin{equation*}
  \nu(X_1=X_2=1)<\infty \quad \text{ and } \quad \nu(X_1=1, Y_{1,1}=Y_{1,2}=1) < \infty. 
  \end{equation*}
  Translating these conditions with the formula \eqref{eq:formulaZneq0}, we recover exactly \eqref{eq:condspecies}.
\end{proof}

Let us now examine $\nu(A\cap\{\mathbf{Z}_b = \mathbf{0}\})$.
{Recall that the event $A\cap\{\mathbf{Z}_b =0\}$ indicates that there are no other merging species blocks than those within the first $b$ blocks.
  The next lemma shows that this implies that we are either in a Kingman-type coalescence (a~pair of species blocks are merging, occurring at rate $a_s$, or a pair of gene blocks within one species are merging, occurring at rate $a_g$), or in a multiple gene coalescence within a single species block (in which case a random proportion $q$ of gene blocks are merging).}

{The key idea is to use exchangeability and the $\sigma$-finiteness property \eqref{eq:nu_conditions} of the measure $\nu$ to show by contradiction that $\nu(A\cap\{\mathbf{Z}_b = \mathbf{0}\})$ is zero in certain cases.}

\begin{lemma}\label{lem:formulaZ=0}
  For an array $(\mathbf{g},\mathbf{s},\mathbf{c})$ satisfying assumptions \eqref{eq:hypothesesarray1} and \eqref{eq:hypothesesarray2}, there exist unique real numbers $a_s, a_g \geq 0$ and a unique measure $\nu_g$ on $(0,1]$ satisfying \eqref{eq:condgenes} such that
  \begin{equation}\label{eq:formulaZ=0}
  \begin{aligned}
  \nu(A\cap\{\mathbf{Z}_b = \mathbf{0}\})&= a_s \1{\left \{k = 2, \mathbf{c}=\mathbf{0}\right \}}\\
  &+ \1{\{k=1\}}\left (a_g \1{\{l_I = 2\}} + \int_{(0,1]}\nu_g(dq)q^{l_I}(1-q)^{g_I-l_I}  \right ),
  \end{aligned}
  \end{equation}
  where $b:=|\mathbf{g}|$, $k:=\sum_i s_i$, $l_i := \sum_j c_{ij}$ and in the case $k=1$, $I$ is the unique index in $\{1, 2,\ldots,b\}$ such that $s_I=1$.
\end{lemma}

\begin{proof}
  The measure $\nu(\mathbf{X}\in \cdot)$ is an exchangeable measure on $\{0,1\}^{\N}$ such that, because of \eqref{eq:nu_conditions}, $\nu (X_1=X_2=X_3=1) < \infty$.
  Note that exchangeability implies that for any $n,i\geq 3$,
  \begin{equation} \label{eq:astuce}
  \nu(\{X_1=X_2=X_3=1\}\cap\{\mathbf{Z}_n = \mathbf{0}\}) =  \nu(\{X_1=X_2=X_i=1\}\cap\{\mathbf{Z}_{i} = \mathbf{0}\}),
  \end{equation}
  But the events $(\{X_1=X_2=X_i=1\}\cap\{\mathbf{Z}_{i} = \mathbf{0}\}, \, i\geq 3)$ are disjoint and all included in $\{X_1=X_2=1\}$, so that
  \[ \sum_{i\geq 3}\nu(\{X_1=X_2=X_i=1\}\cap\{\mathbf{Z}_{i} = \mathbf{0}\}) \leq \nu(\{X_1=X_2=1\}) < \infty.  \]
  From \eqref{eq:astuce} we deduce $\nu(\{X_1=X_2=X_3=1\}\cap\{\mathbf{Z}_n = \mathbf{0}\}) =0$.
  This implies immediately that for a finite array $(\mathbf{g},\mathbf{s}, \mathbf{c}) $ such that $k = \sum_i s_i > 2$, we have $\nu(A\cap\{\mathbf{Z}_b = \mathbf{0}\})=0$.
  \begin{itemize}
    \item In the case $k=2$ (suppose $s_1=s_2=1$), one must examine several cases.
    \begin{itemize}
      \item Suppose first $c_{1,1}=c_{1,2}=1$.
      This means that the first two gene blocks of the first species block coalesce while the first two species blocks coalesce.
      Then we note that for any $n,i\geq 2$,
      \begin{align*}
      &\nu(\{X_1=X_2=1, \, Y_{1,1}=Y_{1,2}=1\}\cap\{\mathbf{Z}_{n} = \mathbf{0}\}) \\
      &= \nu(\{X_1=X_i=1, \, Y_{1,1}=Y_{1,2}=1\}\cap\{\mathbf{Z}_{i} = \mathbf{0}\}).
      \end{align*}
      However,
      \[ \sum_{i\geq 2}\nu(\{X_1=X_i=1, \, Y_{1,1}=Y_{1,2}=1\}\cap\{\mathbf{Z}_{i} = \mathbf{0}\}) \leq \nu(\{Y_{1,1}=Y_{1,2}=1\}) <\infty, \]
      so that necessarily $\nu(\{X_1=X_2=1, \, Y_{1,1}=Y_{1,2}=1\}\cap\{\mathbf{Z}_{n} = \mathbf{0}\})=0$.
      So in the case $c_{1,1}=c_{1,2}=1$, we have $\nu(A\cap\{\mathbf{Z}_b = \mathbf{0}\})=0$.
      
      \item Now suppose $c_{1,1}=c_{2,1}=1$, and all the other $c_{ij}$ are zero.
      From our previous point, note that
      \[ \nu(\{X_1=X_2=1, \, Y_{1,1}=Y_{2,1}=1, \text{ and }\exists j\geq 2, Y_{1,j}=1 \}\cap\{\mathbf{Z}_{n} = \mathbf{0}\}) = 0, \]
      which implies that the events $ (\{X_1=X_2=1, \, Y_{1,j}=Y_{2,1}=1\}\cap{\{\mathbf{Z}_{n} = \mathbf{0}\}},\,{j\geq 1}) $ are $\nu$-a.e.\ disjoint.
      Therefore for any $n\geq 2$,
      \[ \sum_{j\geq 1}\nu(\{X_1=X_2=1, \, Y_{1,j}=Y_{2,1}=1\}\cap\{\mathbf{Z}_{n} = \mathbf{0}\}) \leq \nu(\{X_1=X_2=1\}) <\infty, \]
      So necessarily $\nu(\{X_1=X_2=1, \, Y_{1,j}=Y_{2,1}=1\}\cap\{\mathbf{Z}_{n} = \mathbf{0}\})=0$.
      This implies that in the case $c_{1,1}=c_{2,1}=1$, we have $\nu(A\cap\{\mathbf{Z}_b = \mathbf{0}\})=0$.
      \item The previous two points show that in the case $k=2$, the only way to have $\nu(A\cap\{\mathbf{Z}_b = \mathbf{0}\})>0$ is if $\mathbf{c}=\mathbf{0}$.
      In that case, define
      \begin{align*}\label{eq:truc}
      a_s &:= \nu(\{X_1=X_2=1\}\cap\{\mathbf{Z}_2 = \mathbf{0}\})\\
      &=\nu(\{X_1=X_2=1\}\cap\{\mathbf{Y}=\mathbf{0}\text{ and } \forall k \notin\{1,2\},\,X_k = 0  \}).
      \end{align*}
      Then by exchangeability, for all $i,j\in\N$ with $i\neq j$, we have
      \[ a_s = \nu(\{X_i=X_j=1\}\cap\{\mathbf{Y}=\mathbf{0}\text{ and } \forall k \notin\{i,j\},\,X_k = 0  \}), \]
      and in conclusion, for any array $(\mathbf{g},\mathbf{s}, \mathbf{c})$ such that $k=2$, we have
      \[ \nu(A\cap\{\mathbf{Z}_b = \mathbf{0}\})=\1_{\{\mathbf{c} = \mathbf{0}\}}a_s. \]
    \end{itemize}
    
    \item In the case $k=1$, suppose that $s_1=1$.
    On the event 
    \[ \{X_1=1, \,X_2=X_3=\cdots=X_b=0\}\cap\{\mathbf{Z}_{b}=\mathbf{0}\}, \]
    we have simply $\mathbf{Z}_{1}=\mathbf{0}$, and then the measure
    \[ \nu' := \nu\bigl(\{(Y_{1,j})_{j\in \N} \in \cdot\}\cap\{X_1=1,\,\mathbf{Z}_{1}=\mathbf{0}\}\bigr) \]
    is an exchangeable measure on $\{0,1\}^{\N}$ such that for all $n\in \N$, ${\nu'\left (\sum_{j=1}^{n}Y_{j}\geq 2 \right )<\infty}$.
    Therefore (see for instance Bertoin \cite{bertoinRandom2006}) there exist a unique constant $a_g \geq 0$ and $\nu_g$ a unique measure on $(0,1]$ satisfying \eqref{eq:condgenes} 
    such that $\nu'$ can be written
    \[ \nu'(Y_1 = y_1, Y_2=y_2, \ldots, Y_n=y_n) = a_g\1_{l=2}+\int_{(0,1]}\nu_{g}(dq)q^{l}(1-q)^{n-l}, \]
    for any vector $(y_1, y_2, \ldots, y_n)\in\{0,1\}^{n}\setminus\{\mathbf{0}\}$ such that $l:=\sum_i y_i \geq 2$.
  \end{itemize}
  
  Putting all the previous considerations together yields \eqref{eq:formulaZ=0}. 
\end{proof}

Now it remains to put together Lemma \ref{lem:formulaZneq0} and Lemma \ref{lem:formulaZ=0}.
Recall that we restricted the rate function $q$ to arrays in $\mathcal{C}$, i.e.\ satisfying \eqref{eq:hypothesesarray1} to \eqref{eq:hypothesesarray3}.
The reason for assuming \eqref{eq:hypothesesarray3} is that then we can always write $q_{b,k}(\mathbf{g},\mathbf{s},\mathbf{c})$ as in \eqref{eq:rates_from_nu}, that is
\begin{align*}
q_{b,k}(\mathbf{g},\mathbf{s},\mathbf{c})& = \nu\bigl (\forall 1\leq i\leq b,\,  X_i = s_i,\text{ and }\forall 1\leq j \leq g_i, Y_{ij} = c_{ij}\bigr ) \\ 
&+ \1_{\{\mathbf{c}=\mathbf{0}\}} \nu \left (\forall 1\leq i\leq b,\,  X_i = s_i,\text{ and }\sum_{i=1}^{b} \sum_{j=1}^{g_i} Y_{ij} = 1 \right ).
\end{align*}
Using the previous two lemmas to decompose the two lines on the events $\{\mathbf{Z}_{b}\neq \mathbf{0}\}$ and $\{\mathbf{Z}_{b}= \mathbf{0}\}$, we obtain the formula \eqref{eq:2species}, concluding the proof of Theorem \ref{thm:charac}.

\section{Poissonian construction} \label{sec:poisson}

The goal of the present section is to show how any simple nested exchangeable coalescent can be constructed from a Poisson point process.
Consider two real coefficients $a_s,a_g \geq 0$ and two measures: $\nu_s$ on $E=(0,1]\times \mathcal{M}_1([0,1])$ satisfying \eqref{eq:condspecies}, and $\nu_g$ on $(0,1]$, satisfying \eqref{eq:condgenes}.
Recall the measures $\mathtt{K}_s, \mathtt{K}_g, P^g_{x}$ and $P^s_{x,\mu}$ introduced in Section \ref{sec:statement}, and the measure $\nu(d\mathbf{Z})$ defined on the space $\widehat{E}$ of doubly indexed arrays of 0's and 1's $\mathbf{Z}=(\mathbf{X}, (\mathbf{Y}_i, \, i\geq 1))=(X_i, Y_{ij}, \, i,j\geq 1)$
\[ \nu :=  a_s{\tt K}_s + a_g {\tt K}_g+ \int_{(0,1]}\nu_g (dx)\, P^g_{x} +\int_{(0,1]\times \mathcal{M}_1 ([0,1])}\nu_s(dx,d\mu)\, P^s_{x,\mu} . \]
Note that $\nu$ characterizes the distribution of the SNEC through the relation \eqref{eq:rates_from_nu}.
The key idea of the construction is that $\nu$ necessarily satisfies \eqref{eq:nu_conditions}, which is easily shown using exchangeability and conditions \eqref{eq:condspecies} and \eqref{eq:condgenes}.
First, note that $\nu(\mathbf{Z}=\mathbf{0})=0$ is trivial from our definitions, and that a straightforward union bound yields
\begin{align*}
\nu\Big(\textstyle\sum_{i=1}^n & X_i \geq 2 \text{ or } \exists i \leq n, \textstyle\sum_{j=1}^{n}Y_{ij} \geq 2 \Big) \\
& \leq \sum_{1\leq i < i'\leq n} \nu(X_i=X_{i'}=1) + \sum_{i=1}^{n} \sum_{1\leq j< j'\leq n} \nu(X_i=Y_{ij}=Y_{ij'}=1) \\
&= \frac{n(n-1)}{2}\nu(X_1=X_2=1) + \frac{n^2(n-1)}{2}\nu(X_1=Y_{1,1}=Y_{1,2}=1),
\end{align*}
therefore we need only check that these two quantities are finite.
Now by definition, we have
\begin{align*}
\mathtt{K}_s(X_1=X_2=1)=1, &\quad \mathtt{K}_s(X_1=Y_{1,1}=Y_{1,2}=1)=0,\\
\mathtt{K}_g(X_1=X_2=1)=0, &\quad \mathtt{K}_g(X_1=Y_{1,1}=Y_{1,2}=1)=1,\\
P^{g}_{x}(X_1=X_2=1)=0, &\quad P^{g}_{x}(X_1=Y_{1,1}=Y_{1,2}=1)=x^2,\\
P^{s}_{x,\mu}(X_1=X_2=1)=x^2, &\quad P^{s}_{x,\mu}(X_1=Y_{1,1}=Y_{1,2}=1)=x\int_{[0,1]}\mu(dq)q^2.
\end{align*}
Integrating the last two lines with respect to $\nu_g$ and $\nu_s$ and summing, we see that \eqref{eq:condspecies} and \eqref{eq:condgenes} imply that both $\nu(X_1=X_2=1)$ and $\nu(X_1=Y_{1,1}=Y_{1,2}=1)$ are finite, proving \eqref{eq:nu_conditions}.

Now to start the construction of our process, consider an initial partition $\pi_0\in\mathcal{N}_\infty$.
Let $M$ be a Poisson point process on $(0,\infty) \times \widehat{E}$ with intensity $dt \otimes \nu(d\mathbf{Z})$.
We will construct on the same probability space the processes $\mathcal{R}^{n}= (\mathcal{R}^{n}(t), \, t\geq 0)$, for $n\in\mathbb{N}$ thanks to $M$.

Recall that for any $\mathbf{Z}=(\mathbf{X}, (\mathbf{Y}_i, \, i\geq 1))=(X_i, Y_{ij}, \, i,j\geq 1)$ and any nested partition $\pi \in\mathcal{N}_n$, we denote by $C(\pi, \mathbf{Z})$ the nested partition of $\mathcal{N}_n$ obtained from $\pi$ by merging exactly the blocks that \emph{participate in the coalescence} where 
\begin{itemize}
  \item The $i$-th block of $\pi^s$ participates iff $X_i=1$;
  \item The $j$-th block in $\pi^g$ of the $i$-th block of $\pi^s$ participates iff $X_i=1$ and $Y_{ij}=1$.   
\end{itemize}
Fix $n\in\N$, and let $M_n$ denote the subset of $M$ consisting of points $(t, \mathbf{Z})$ such that $\sum_{i=1}^{n}X_i\geq 2$ or $\exists i \leq n, \, X_i\sum_{j=1}^{n}Y_{ij}\geq 2$.
Because of \eqref{eq:nu_conditions}, there are only a finite number of such points with $t$ in a compact set of $[0,+\infty)$.
Therefore one can label the atoms of the set $M_n:=\{ ( t_k, \mathbf{Z}^{ (k) } ), k\in\N \}$ in increasing order, i.e.\ such that $0 \leq t_1\leq t_2 \ldots$

We set $\mathcal{R}^{n} (t)= (\pi_0)_{|n} $ for $t\in[0,t_1)$.
Then define recursively
\[
\mathcal{R}^{n} (t)  = C(  \mathcal{R}^{n} (t_i-),  \mathbf{Z}^{(i)}  ),  \quad \text{for every } t\in[t_i,t_{i+1}). 
\]
These processes are consistent in $n$ as we show in the following result.

\begin{proposition}
  \label{prop:Poisson}
  For every $t\geq0$, the sequence of random bivariate partitions $(\mathcal{R}^{n}(t), n\in\mathbb{N} )$ is consistent. 
  If we denote by $\mathcal{R}(t)$ the unique partition of $\mathcal{N}_\infty$ such that $\mathcal{R}_{ \vert n }(t) =  \mathcal{R}^{n}(t)$ for every $n\in\mathbb{N}$, then the process $\mathcal{R}=( \mathcal{R}(t), t\geq0 )$ is a SNEC started from $\pi_0$, with rates given as in Theorem \ref{thm:charac}.
\end{proposition}

The proof uses similar arguments as in the proof of consistency of exchangeable coalescents given in Proposition 4.5 of \cite{bertoinRandom2006}.

\begin{proof}
  The key idea (basically (4.4) in \cite{bertoinRandom2006}) is that by definition, the coagulation operator satisfies
  \begin{equation}\label{BFragCoagEq4.4}
  \text{Coag}_2(\pi, \tilde\pi)_{|n} = \text{Coag}_2(\pi_{|n}, \tilde\pi) = \text{Coag}_2(\pi_{|n}, \tilde\pi_{|n})
  \end{equation}
  for any $\pi, \tilde{ \pi }$ and $n$ for which this is well defined.
  
  Recall that we defined $M_n$ as the subset of $M$ consisting of points $(t, \mathbf{Z})$ such that $\sum_{i=1}^{n}X_i\geq 2$ or $\exists i \leq n, \, X_i\sum_{j=1}^{n}Y_{ij}\geq 2$.
  Fix $n\geq2$ and write $( t_1, \mathbf{Z}^{(1)} )$ for the first atom of $M_n$ on $(0,\infty) \times \widehat{E}$.
  Plainly, $\mathcal{R}^{ n-1 }(t) = \mathcal{R}^{n}_{ | n-1}(t) = (\pi_{0})_{|n-1}$ for every $t\in[0,t_1)$.
  
  Consider first the case when $\sum_{i=1}^{n-1}X^{(1)}_i\geq 2$ or $\exists i \leq n-1, \, X^{(1)}_i\sum_{j=1}^{n-1}Y^{(1)}_{ij}\geq 2$.
  Then $( t_1, \mathbf{Z}^{(1)} )$ is also the first atom of $M_{n-1}$ and by definition and using \eqref{BFragCoagEq4.4}, $\mathcal{R}^{ n-1 }(t_1)=\mathcal{R}^{n}_{ |n-1 }(t)$.
  
  Now suppose $\sum_{i=1}^{n-1}X^{(1)}_i\leq 1$ and $\forall i \leq n-1, \, X^{(1)}_i\sum_{j=1}^{n-1}Y^{(1)}_{ij}\leq 1$.
  This implies that at time $t_1$, there is no species (resp.\ genes) coalescence between the $n-1$ first species (resp.\ genes) of $\mathcal{R}^{n}(t_1 -)$.
  Therefore the coalescence event in $\mathcal{R}^{n}$ at time $t_1$ leaves the first $n-1$ blocks of $\mathcal{R}^{n}(t_1 -)^{s}$ or $\mathcal{R}^{n}(t_1 -)^{g}$ unchanged, though there may be a coalescence involving the $n$-th block (in that case, necessarily a singleton $\{n\}$) and one of the $n-1$ first blocks.
  So finally $\mathcal{R}^{n}(t_1)_{|n-1}= \mathcal{R}^{n}(t_1 -)_{|n-1}=\mathcal{R}^{n-1}(t_1 )$.
  
  In both cases we have $\mathcal{R}^{n}(t_1)_{|n-1}=\mathcal{R}^{n-1}(t_1 )$, and by an obvious induction this is true for any further jump of the process $\mathcal{R}^{n}$, so that for all $t\geq 0$,
  \[ \mathcal{R}^{n}(t)_{|n-1}=\mathcal{R}^{n-1}(t). \]
  This shows the existence of $\mathcal{R}$ such that for all $n$, $\mathcal{R}_{|n}=\mathcal{R}^{n}$.
  
  From this Poissonian construction $\mathcal{R}^{n}$ is a Markov process, and by definition the arrays $\mathbf{Z}^{(i)}_{|[n]^2}$ are hierarchically exchangeable, which implies that $\mathcal{R}^{n}$ is an exchangeable process.
  Clearly by construction $\mathcal{R}^n(t)$ is nested for all $t$, and the only jumps of the process $\mathcal{R}^n$ are coalescence events.
  According to Lemma~\ref{lem:projMarkov}, the process $\mathcal{R}$ is a SNEC process.
  Because the arrays $\mathbf{Z}$, where $(t,\mathbf{Z})\in M$, are the same arrays that appear in the proof of Theorem \ref{thm:charac}, it is clear that the jump rates of $\mathcal{R}^{n}$ are those given in Theorem~\ref{thm:charac}.
\end{proof}

\section{Marginal coalescents -- Coming down from infinity}\label{sec:CDI}

Consider a SNEC process $\mathcal{R} =(\mathcal{R}^s, \mathcal{R}^g)$, with rates given as in Theorem \ref{thm:charac} by two coefficients $a_s, a_g \geq 0$ and two measures, $\nu_s$ on $E=(0,1]\times\mathcal M_1([0,1])$ and $\nu_g$ on $(0,1]$ satisfying \eqref{eq:condspecies} and \eqref{eq:condgenes}.

It is obvious from Proposition \ref{prop:Poisson}  that $(\mathcal{R}^s(t),\; t\geq 0)$ is a simple coalescent process, with Kingman coefficient $a_s$ and coagulation measure $\widehat{\nu}_s$ satisfying \eqref{eqn:condLambda} which is the push-forward of $\nu_s(dp,d\mu)$ by the application $(p,\mu) \mapsto p$.
Let us call this univariate coalescent the \emph{(marginal) species coalescent} of the SNEC process $\mathcal{R}$.

Now, notice that under an initial condition with a unique species block (i.e., $\mathcal{R}^s$ is constant to the coarsest partition  $\mathbf{1}_\infty$), the process $(\mathcal{R}^g(t),\; t\geq 0)$ also behaves as a simple coalescent process, with Kingman coefficient $a_g$ and coagulation measure $\widehat{\nu}_g$ defined by
\[ \forall B \text{ Borel set of }(0,1], \quad \widehat{\nu}_g(B) := \nu_g(B) + \int_{(0,1]\times \mathcal{M}_1([0,1])}\nu_s(dp,d\mu)p\,\mu(B). \]
We call the simple coalescent thus defined the \emph{(marginal) gene coalescent} of the SNEC process $\mathcal{R}$.

Equivalently, in terms of $\Lambda$-coalescents, the marginal species coalescent is a $\Lambda_s$-coalescent with $\Lambda_s$ defined by
\begin{equation}\label{eq:marg_gene}
\forall B \text{ Borel set of }[0,1], \quad \Lambda_s(B) = a_s \delta_0(B) + \int_{B\times \mathcal{M}_1([0,1])}\nu_s(dp,d\mu) p^2, 
\end{equation}
and the marginal gene coalescent is a $\Lambda_g$-coalescent with $\Lambda_g$ defined for all $B$ Borel set of $[0,1]$ by
\begin{equation}\label{eq:marg_species}
\Lambda_g(B) = a_g \delta_0(B) + \int_{B}\nu_g(dq) q^2 + \int_{(0,1]\times \mathcal{M}_1([0,1])}\nu_s(dp,d\mu)p\int_{B}\mu(dq) q^2.
\end{equation}

These two marginal processes allow us to express properties of the initial bivariate SNEC process.
Consider an initial state $\rho_0\in\mathcal{N}_\infty$ with infinitely many species blocks, each containing infinitely many gene blocks.
In a way analogous to the one-dimensional case, recalling that $\abs{\mathcal{R}^g(t)} \geq \abs{\mathcal{R}^s(t)}$ for all $t\geq 0$, we will say that a SNEC \emph{comes down from infinity} (CDI) if for all $t>0$
\[
\abs{\mathcal{R}^{g}(t)} < \infty  \qquad \P_{\rho_0}\text{-a.s.}
\] 
In the univariate case, characterizing which coalescent processes come down from infinity has been solved \cite{schweinsbergNecessary2000} for $\Lambda$-coalescents, with the following necessary and sufficient condition for coming down from infinity:
\begin{equation*} \label{eq:conditionCDI}
\sum_{n\geq 2}\left( \sum_{k=2}^{n}(k-1)\binom{n}{k} \int_{[0,1]}\Lambda(dp) p^{k-2}(1-p)^{n-k} \right)^{-1} < \infty. 
\end{equation*}
Note that the previous condition is true as soon as $\Lambda$ has an atom at $0$ ($\Lambda(\{0\})$ is the Kingman coefficient of the process). An equivalent criterion (see \cite{BLG06}, and \cite{BBL2014} for a probabilistic proof) is the integrability of $1/\psi$ near $+\infty$, where
\begin{equation}\label{eq:CDIcriterion}
\psi (q) :=\int_{[0,1]}\left(e^{-qx}-1+qx\right)x^{-2}\,\Lambda(dx).
\end{equation}
We will now see that in the case of simple nested coalescents, we can give a general characterization of the different CDI properties of a SNEC process, depending only on the properties of the marginal species and marginal gene coalescents.

First notice that if the marginal gene coalescent does not CDI, then any species block with infinitely many gene blocks at some time $t$ clearly keeps infinitely many gene blocks for any $t'\geq t$.
Also in any case the process $\mathcal{R}^s$ has the distribution of the marginal species coalescent, so determining whether the number of species comes down from infinity is trivial.

\begin{proposition} \label{prop:CDI}
  We assume here that $\widehat{\nu}_s(\{1\})=\widehat{\nu}_g(\{1\})=0$ and that the marginal gene coalescent comes down from infinity (CDI). Then we have the following three cases.
  \begin{enumerate}[i)]
    \item If the marginal species coalescent CDI as well, then $ \mathcal{R} $ CDI.
    \item If the marginal species coalescent does not CDI but $\displaystyle\int_{[0,1]} \widehat{\nu}_s(dx)\,x=\infty$, then for any initial condition with infinitely many species blocks and for each time $t>0$, the number of gene blocks in each species block of $\mathcal{R}(t)$ is infinite a.s.
    \item If the marginal species coalescent does not CDI and $\displaystyle\int_{[0,1]} \widehat{\nu}_s(dx)\,x<\infty$, then for any initial condition and for each time $t>0$, the number of gene blocks in each species block of $\mathcal{R}(t)$ is finite a.s.
  \end{enumerate}
\end{proposition}
As a consequence of this proposition, it is clear that $\mathcal{R}$ comes down from infinity if and only if both the marginal species coalescent and the marginal gene coalescent come down from infinity.

A simple example of a SNEC process coming down from infinity is the nested Kingman coalescent (`Kingman in Kingman'), given by its marginal rates $a_s,a_g>0$, defined so that each pair of species coalesces at rate $a_s$ independently of the others, and each pair of genes within the same species coalesces at rate $a_g$ independently of the rest.
Since the marginal coalescents are precisely two Kingman coalescents, they both come down from infinity.

Note that the Bolthausen-Sznitman coalescent \cite{BS1998} (denoted $U$-coalescent in \cite{pitmanCoalescents1999} because the measure $\Lambda$ is uniform on $[0,1]$) satisfies the conditions of the peculiar case \textit{ii)}.
So for a SNEC $\mathcal{R}$ defined by a Kingman gene coalescent evolving within a species $U$-coalescent, at each positive time the number of gene blocks within a species block is infinite (if the initial state $\rho_0$ has an infinite number of species blocks).

Case \textit{iii)} can easily be obtained by considering a ``slow'' species coalescent, such as a $\delta_x$-coalescent for $x\in(0,1)$, or any $\beta(a,b)$-coalescent with $a>1, b>0$ (that is a $\Lambda$-coalescent with $\Lambda(dx) = c_{a,b} x^{a-1}(1-x)^{b-1} dx$).

\begin{proof}
  \textit{\textbf{i)}} Suppose both marginal coalescents come down from infinity, and consider an initial state $\rho\in\mathcal{N}_\infty$ with infinitely many species blocks, each containing infinitely many gene blocks.
  
  Choose $t>0$.
  Since $\mathcal{R}^{s}$ comes down from infinity, we have $\P_{\rho}(\abs{\mathcal{R}^{s}(t/2)}<\infty)=1$, and necessarily, $\mathcal{R}^{s}$ stays constant on an interval $[t/2, T[$, where $T$ is its next jump time.
  Now on the interval $[t/2, \min(T,t)[$, within each of the $ \abs{\mathcal{R}^{s}(t/2)} $ species block, the gene blocks undergo independent coalescent processes which CDI, therefore there are finitely many gene blocks in each species at time $\min(T,t)$, which implies
  \[ 
  \P_{\rho}(\abs{\mathcal{R}^{g}(t)} < \infty) = 1.
  \]
  Let us say a few words before proving \textit{ii)} and \textit{iii)}. 
  Pick $t>0$ and focus on the species containing $1$ (the first species).
  To this aim, write $M(t)$ for the number of genes within the first species, at time $t$.
  By exchangeability, to show \textit{ii)} it is sufficient to show $ \P_{\rho}(M(t) = \infty)=1 $, for any initial condition $\rho$ with infinitely many species blocks, and to show \textit{iii)} it is sufficient to show $ \P_{\rho}(M(t) < \infty)=1 $, for any initial condition $\rho$.
  
  \textit{\textbf{ii)}} Suppose $\int_{[0,1]} \widehat{\nu}_s(dx)\,x=\infty$.
  First, note that since the species coalescent does not CDI and $\widehat{\nu}_s(\{1\})=0$, there are at all times $t\geq 0$ infinitely many species blocks (see for instance \cite[Proposition 23]{pitmanCoalescents1999}).
  Now let us fix $\delta \in(0,t]$ and $\epsilon \in (0,1]$, and investigate the random number of coalescence events in the time interval $[t-\delta,t]$ involving the first species and at least a proportion $\epsilon$ of all other species.
  More precisely, we consider the number of atoms $(s, \mathbf{Z})$ in the Poissonian construction such that $s\in [t-\delta,t]$, $X_1 = 1$ and $\lim_{n\to\infty}\sum_{i=1}^{n} X_i/n \geq \epsilon $.
  From the Poissonian construction, it is easy to see that this number is a Poisson random variable with mean
  \[ \delta \int_{[\epsilon,1]} \widehat{\nu}_s(dx)\,x. \]
  Pick any $A\in\N$ and $\eta>0$.
  We will show $\P_{\rho}(M(t)\leq A) < 2 \eta$, which is sufficient to conclude that $M(t)=\infty$ a.s.
  Note that we assumed that the marginal gene coalescent CDI, so for $\Pi=(\Pi(t), \,t\geq 0)$ a version of this univariate coalescent started from $\mathbf{0}_{\infty}$, we have $\P(\abs{\Pi(\delta)}<\infty)=1$ for all $\delta >0$. In addition, $\Pi$ is right-continuous, so $\abs{\Pi(\delta)} \uparrow\infty$ as $\delta\to 0$.
  Therefore, one can choose $\delta > 0$ small enough, and then $\epsilon >0$ such that
  \begin{equation}\label{eq:approxii}
  \P(\abs{\Pi(\delta)} \leq A) < \eta \quad \text{ and } \quad e(\epsilon) :=  \int_{[\epsilon,1]} \widehat{\nu}_s(dx)\,x \geq \frac{-\log (\eta)}{\delta}.
  \end{equation}
  Now consider the stopping time defined by
  \[\begin{aligned}
  T := \inf\{s \geq t-\delta, \; \text{the first species participates at time $s$ in a coalescence event }\\
  \text{involving at least a proportion $\epsilon$ of other species}\}.
  \end{aligned} \]
  By the Poisson construction, $T-(t-\delta)$ is an exponential random variable with parameter $e(\epsilon)$, so from \eqref{eq:approxii} we deduce
  \[ \P_{\rho}(T \geq t) \leq \eta. \]
  Now since $T$ is a coalescence time for the first species, we have $M(T)=\infty$ almost surely.
  Indeed, the assumption $\widehat{\nu}^{g}(\{1\})=0$ implies that not every gene participates in the coalescence.
  But since an infinite number of species participate in the coalescence, the law of large numbers implies that in the newly formed species, there is an infinite number of genes which do not coalesce at time $T$.
  Since $M(T)=\infty$, we can define a random injection $\sigma : \N \to \N$ mapping $k$ to the first element of the $k$-th gene of the first species at time $T$.
  We then define $\widetilde{\Pi}(u) := \sigma(\mathcal{R}^g(T+u))$, which has by the strong Markov property the distribution of a marginal gene coalescent started from $\mathbf{0}_{\infty}$, independent of $T$.
  Furthermore, by construction we have $M(T+u) \geq \abs{\widetilde{\Pi}(u)}$ a.s., so that finally
  \begin{align*}
  \P_{\rho}(M(t)\leq A) &\leq \P_{\rho}(T>t)+\P_{\rho}(t-\delta \leq T\leq t)\,\P_\rho(\abs{\widetilde{\Pi}(t-T)}\leq A \mid t-\delta \leq T\leq t)\\
  &\leq \P_{\rho}(T>t)+\P(\abs{\Pi(\delta)}\leq A)\\
  &\leq 2\eta.
  \end{align*}
  
  \textit{\textbf{iii)}} Now supposing $\int_{[0,1]} \widehat{\nu}_s(dx)\,x<\infty$, with the same argument as previously, the first species participates in coalescence events at some random times $0<T_1<T_2<\ldots$, distributed as a Poisson process with parameter $\int_{[0,1]} \widehat{\nu}_s(dx)\,x$, and all these events involve infinitely many species blocks (recall the marginal species coalescent does not CDI and so in particular has $a_s=0$).
  Let $T_0:=0$ by convention and for each $i$, we can define a random injection $\sigma_i : \N \to \N$ mapping $k$ to the first element of the $k$-th gene of the first species at time $T_i$.
  Now because the first species does not change during the intervals $[T_i, T_{i+1})$, the process $\widetilde{\Pi}_i$ defined by
  \[ \widetilde{\Pi}_i(u) := \sigma_i(\mathcal{R}^g(T_i+u)) \]
  is a marginal gene coalescent (and so CDI by assumption), which is independent of $T_i$, and there is the following equality between processes, for $u < T_{i+1}-T_i$,
  \[ M(T_i+u)=\widetilde{\Pi}_i(u). \]
  Finally, we have for any $t>0$, and any initial $\rho\in\mathcal{N}_\infty$,
  \begin{align*}
  \P_{\rho}(M(t)<\infty) &= \sum_{i\geq 0}\P_{\rho}(T_i < t < T_{i+1})\,\P_{\rho}(M(t)<\infty\mid T_i < t < T_{i+1})\\
  &= \sum_{i\geq 0}\P_{\rho}(T_i < t < T_{i+1})\,\P_{\rho}(\widetilde{\Pi}_i(t-T_i)<\infty \mid T_i < t < T_{i+1})\\
  &= \sum_{i\geq 0}\P_{\rho}(T_i < t < T_{i+1}) = 1,
  \end{align*}
  which concludes the proof.
\end{proof}

\paragraph{\bf Acknowledgments.} The four authors thank the {\em Center for Interdisciplinary Research in Biology} (Coll\`ege de France) for travel funding.
ABB and ASJ would like to thank CIMAT, A.C. and especially Víctor Rivero for support and for his comments on this project, which started during Airam's PhD thesis.
ABB is supported by CONACYT postdoctoral grant 234823, and ASJ by CONACYT grant CB-2014/243068.

%

\begin{thebibliography}{45}
  \providecommand{\natexlab}[1]{#1}
  \providecommand{\url}[1]{\texttt{#1}}
  \expandafter\ifx\csname urlstyle\endcsname\relax
  \providecommand{\doi}[1]{doi: #1}\else
  \providecommand{\doi}{doi: \begingroup \urlstyle{rm}\Url}\fi
  
  \bibitem[Berestycki et~al.(2013)Berestycki, Berestycki, and
  Schweinsberg]{BBS2013}
  J.~Berestycki, N.~Berestycki, and J.~Schweinsberg.
  \newblock The genealogy of branching {{Brownian}} motion with absorption.
  \newblock \emph{The Annals of Probability}, 41\penalty0 (2):\penalty0 527--618,
  Mar. 2013.
  \newblock \doi{10.1214/11-AOP728}.
  
  \bibitem[Berestycki et~al.(2014)Berestycki, Berestycki, and Limic]{BBL2014}
  J.~Berestycki, N.~Berestycki, and V.~Limic.
  \newblock A small-time coupling between ${\Lambda}$-coalescents and branching
  processes.
  \newblock \emph{The Annals of Applied Probability}, 24\penalty0 (2):\penalty0
  449--475, Apr. 2014.
  \newblock \doi{10.1214/12-AAP911}.
  
  \bibitem[Berestycki(2009)]{Ber2009}
  N.~Berestycki.
  \newblock Recent progress in coalescent theory.
  \newblock \emph{Ensaios Matem{\'a}ticos}, 16\penalty0 (1):\penalty0 1--193,
  2009.
  
  \bibitem[Bertoin(2006)]{bertoinRandom2006}
  J.~Bertoin.
  \newblock \emph{Random Fragmentation and Coagulation Processes}.
  \newblock {Cambridge University Press}, Cambridge, 2006.
  \newblock \doi{10.1017/CBO9780511617768}.
  
  \bibitem[Bertoin and Le~Gall(2003)]{BLG03}
  J.~Bertoin and J.-F. Le~Gall.
  \newblock Stochastic flows associated to coalescent processes.
  \newblock \emph{Probability Theory and Related Fields}, 126\penalty0
  (2):\penalty0 261--288, 2003.
  
  \bibitem[Bertoin and Le~Gall(2006)]{BLG06}
  J.~Bertoin and J.-F. Le~Gall.
  \newblock Stochastic flows associated to coalescent processes. {III}. {L}imit
  theorems.
  \newblock \emph{Illinois Journal of Mathematics}, 50\penalty0 (1-4):\penalty0
  147--181, 2006.
  
  \bibitem[Blancas and Wakolbinger()]{airam}
  A.~Blancas and A.~Wakolbinger.
  \newblock A representation for the semigroup of a two-level {Fleming-Viot}
  process in terms of the {Kingman} nested coalescent.
  \newblock \emph{In preparation}.
  
  \bibitem[Blancas et~al.(2018)Blancas, Rogers, Schweinsberg, and
  Siri-J{\'e}gousse]{airam_arno}
  A.~Blancas, T.~Rogers, J.~Schweinsberg, and A.~Siri-J{\'e}gousse.
  \newblock The nested {K}ingman coalescent: Speed of coming down from infinity.
  \newblock \emph{arXiv:1803.08973 [math]}, Mar. 2018.
  
  \bibitem[Bolthausen and Sznitman(1998)]{BS1998}
  E.~Bolthausen and A.-S. Sznitman.
  \newblock On {{Ruelle}}'s probability cascades and an abstract cavity method.
  \newblock \emph{Communications in Mathematical Physics}, 197\penalty0
  (2):\penalty0 247--276, 1998.
  \newblock \doi{10.1007/s002200050450}.
  
  \bibitem[Brunet and Derrida(2013)]{BD13}
  {\'E}.~Brunet and B.~Derrida.
  \newblock Genealogies in simple models of evolution.
  \newblock \emph{Journal of Statistical Mechanics: Theory and Experiment},
  2013\penalty0 (01):\penalty0 P01006, 2013.
  
  \bibitem[Dawson(2018)]{dawsonMultilevel2018}
  D.~A. Dawson.
  \newblock Multilevel mutation-selection systems and set-valued duals.
  \newblock \emph{Journal of Mathematical Biology}, 76\penalty0 (1-2):\penalty0
  295--378, Jan. 2018.
  \newblock \doi{10.1007/s00285-017-1145-2}.
  
  \bibitem[Degnan and Rosenberg(2009)]{DR09}
  J.~H. Degnan and N.~A. Rosenberg.
  \newblock Gene tree discordance, phylogenetic inference and the multispecies
  coalescent.
  \newblock \emph{Trends in ecology \& evolution}, 24\penalty0 (6):\penalty0
  332--340, 2009.
  
  \bibitem[Desai et~al.(2013)Desai, Walczak, and Fisher]{DWF13}
  M.~M. Desai, A.~M. Walczak, and D.~S. Fisher.
  \newblock Genetic diversity and the structure of genealogies in rapidly
  adapting populations.
  \newblock \emph{Genetics}, 193\penalty0 (2):\penalty0 565--585, 2013.
  
  \bibitem[Doyle(1997)]{D97}
  J.~J. Doyle.
  \newblock Trees within trees: genes and species, molecules and morphology.
  \newblock \emph{Systematic Biology}, 46\penalty0 (3):\penalty0 537--553, 1997.
  
  \bibitem[Duchamps(2018)]{D}
  J.-J. Duchamps.
  \newblock Trees within trees {{II}}: {{Nested Fragmentations}}.
  \newblock \emph{arXiv:1807.05951}, July 2018.
  
  \bibitem[Durrett and Schweinsberg(2005)]{DS05}
  R.~Durrett and J.~Schweinsberg.
  \newblock A coalescent model for the effect of advantageous mutations on the
  genealogy of a population.
  \newblock \emph{Stochastic processes and their applications}, 115\penalty0
  (10):\penalty0 1628--1657, 2005.
  
  \bibitem[Eldon and Wakeley(2006)]{EW2006}
  B.~Eldon and J.~Wakeley.
  \newblock Coalescent processes when the distribution of offspring number among
  individuals is highly skewed.
  \newblock \emph{Genetics}, 172\penalty0 (4):\penalty0 2621--2633, Apr. 2006.
  \newblock \doi{10.1534/genetics.105.052175}.
  
  \bibitem[Etheridge(2011)]{Eth2011}
  A.~Etheridge.
  \newblock \emph{Some Mathematical Models from Population Genetics:
    {{{\'E}cole}} d'{\'e}t{\'e} de Probabilit{\'e}s de {Saint}-{Flour
      XXXIX}-2009}.
  \newblock Lecture notes in mathematics. {Springer}, 2011.
  
  \bibitem[Felsenstein(2004)]{F04}
  J.~Felsenstein.
  \newblock \emph{Inferring phylogenies}, volume~2.
  \newblock Sinauer associates Sunderland, MA, 2004.
  
  \bibitem[Foutel-Rodier et~al.(2018)Foutel-Rodier, Lambert, and Schertzer]{FRLS}
  F.~Foutel-Rodier, A.~Lambert, and E.~Schertzer.
  \newblock Exchangeable coalescents, ultrametric spaces, nested
  interval-partitions: {{A}} unifying approach.
  \newblock \emph{arXiv:1807.05165 [math]}, July 2018.
  
  \bibitem[Grenfell et~al.(2004)Grenfell, Pybus, Gog, Wood, Daly, Mumford, and
  Holmes]{GPG04}
  B.~T. Grenfell, O.~G. Pybus, J.~R. Gog, J.~L. Wood, J.~M. Daly, J.~A. Mumford,
  and E.~C. Holmes.
  \newblock Unifying the epidemiological and evolutionary dynamics of pathogens.
  \newblock \emph{Science}, 303\penalty0 (5656):\penalty0 327--332, 2004.
  
  \bibitem[Heled and Drummond(2009)]{HD09}
  J.~Heled and A.~J. Drummond.
  \newblock Bayesian inference of species trees from multilocus data.
  \newblock \emph{Molecular biology and evolution}, 27\penalty0 (3):\penalty0
  570--580, 2009.
  
  \bibitem[Kallenberg(2005)]{kallenbergProbabilistic2005}
  O.~Kallenberg.
  \newblock \emph{Probabilistic Symmetries and Invariance Principles}.
  \newblock Probability and Its Applications. {Springer-Verlag}, New York, 2005.
  \newblock \doi{10.1007/0-387-28861-9}.
  
  \bibitem[Kingman(1982)]{K82}
  J.~Kingman.
  \newblock The coalescent.
  \newblock \emph{Stochastic processes and their applications}, 13\penalty0
  (3):\penalty0 235--248, 1982.
  
  \bibitem[Lambert(2008)]{Lam08}
  A.~Lambert.
  \newblock Population dynamics and random genealogies.
  \newblock \emph{Stochastic Models}, 24\penalty0 (sup1):\penalty0 45--163, 2008.
  \newblock \doi{10.1080/15326340802437728}.
  
  \bibitem[Lambert(2017{\natexlab{a}})]{L18}
  A.~Lambert.
  \newblock Random ultrametric trees and applications.
  \newblock \emph{ESAIM: Procs}, 60:\penalty0 70--89, 2017{\natexlab{a}}.
  \newblock \doi{10.1051/proc/201760070}.
  
  \bibitem[Lambert(2017{\natexlab{b}})]{Lam17}
  A.~Lambert.
  \newblock Probabilistic models for the (sub)tree(s) of life.
  \newblock \emph{Braz. J. Probab. Stat.}, 31\penalty0 (3):\penalty0 415--475, 08
  2017{\natexlab{b}}.
  \newblock \doi{10.1214/16-BJPS320}.
  
  \bibitem[Lambert and Schertzer(2018)]{LS}
  A.~Lambert and E.~Schertzer.
  \newblock Coagulation-transport equations and the nested coalescents.
  \newblock \emph{arXiv:1807.09153}, July 2018.
  
  \bibitem[Maddison(1997)]{M97}
  W.~P. Maddison.
  \newblock Gene trees in species trees.
  \newblock \emph{Systematic biology}, 46\penalty0 (3):\penalty0 523--536, 1997.
  
  \bibitem[Matuszewski et~al.(2017)Matuszewski, Hildebrandt, Achaz, and
  Jensen]{MHA17}
  S.~Matuszewski, M.~E. Hildebrandt, G.~Achaz, and J.~D. Jensen.
  \newblock Coalescent processes with skewed offspring distributions and
  non-equilibrium demography.
  \newblock \emph{Genetics}, pages genetics--300499, 2017.
  
  \bibitem[Neher and Hallatschek(2013)]{NH13}
  R.~A. Neher and O.~Hallatschek.
  \newblock Genealogies of rapidly adapting populations.
  \newblock \emph{Proceedings of the National Academy of Sciences}, 110\penalty0
  (2):\penalty0 437--442, 2013.
  \newblock \doi{10.1073/pnas.1213113110}.
  
  \bibitem[Nei and Kumar(2000)]{NK00}
  M.~Nei and S.~Kumar.
  \newblock \emph{Molecular evolution and phylogenetics}.
  \newblock Oxford university press, 2000.
  
  \bibitem[Page and Charleston(1997)]{PC97}
  R.~D. Page and M.~A. Charleston.
  \newblock From gene to organismal phylogeny: reconciled trees and the gene
  tree/species tree problem.
  \newblock \emph{Molecular phylogenetics and evolution}, 7\penalty0
  (2):\penalty0 231--240, 1997.
  
  \bibitem[Page and Charleston(1998)]{PC98}
  R.~D. Page and M.~A. Charleston.
  \newblock Trees within trees: phylogeny and historical associations.
  \newblock \emph{Trends in Ecology \& Evolution}, 13\penalty0 (9):\penalty0
  356--359, 1998.
  
  \bibitem[Pitman(1999)]{pitmanCoalescents1999}
  J.~Pitman.
  \newblock Coalescents with multiple collisions.
  \newblock \emph{The Annals of Probability}, 27\penalty0 (4):\penalty0
  1870--1902, Oct. 1999.
  \newblock \doi{10.1214/aop/1022874819}.
  
  \bibitem[Rosenberg(2002)]{R02}
  N.~A. Rosenberg.
  \newblock The probability of topological concordance of gene trees and species
  trees.
  \newblock \emph{Theoretical population biology}, 61\penalty0 (2):\penalty0
  225--247, 2002.
  
  \bibitem[Sagitov(1999)]{sagitovGeneral1999}
  S.~Sagitov.
  \newblock The general coalescent with asynchronous mergers of ancestral lines.
  \newblock \emph{Journal of Applied Probability}, 36\penalty0 (4):\penalty0
  1116--1125, Dec. 1999.
  \newblock \doi{10.1239/jap/1032374759}.
  
  \bibitem[Schweinsberg(2000{\natexlab{a}})]{schweinsbergCoalescents2000}
  J.~Schweinsberg.
  \newblock Coalescents with simultaneous multiple collisions.
  \newblock \emph{Electronic Journal of Probability}, 5, 2000{\natexlab{a}}.
  \newblock \doi{10.1214/EJP.v5-68}.
  
  \bibitem[Schweinsberg(2000{\natexlab{b}})]{schweinsbergNecessary2000}
  J.~Schweinsberg.
  \newblock A necessary and sufficient condition for the {$\Lambda$}-coalescent
  to come down from infinity.
  \newblock \emph{Electronic Communications in Probability}, 5:\penalty0 1--11,
  2000{\natexlab{b}}.
  \newblock \doi{10.1214/ECP.v5-1013}.
  
  \bibitem[Schweinsberg(2003)]{Sch2003}
  J.~Schweinsberg.
  \newblock Coalescent processes obtained from supercritical {Galton}-{Watson}
  processes.
  \newblock \emph{Stochastic Processes and their Applications}, 106\penalty0
  (1):\penalty0 107--139, July 2003.
  \newblock \doi{10.1016/S0304-4149(03)00028-0}.
  
  \bibitem[Schweinsberg(2017)]{S17}
  J.~Schweinsberg.
  \newblock Rigorous results for a population model with selection {II}:
  genealogy of the population.
  \newblock \emph{Electronic Journal of Probability}, 22, 2017.
  
  \bibitem[Semple and Steel(2003)]{SS03}
  C.~Semple and M.~A. Steel.
  \newblock \emph{Phylogenetics}, volume~24.
  \newblock Oxford University Press, 2003.
  
  \bibitem[Sz{\"o}ll{\H{o}}si et~al.(2014)Sz{\"o}ll{\H{o}}si, Tannier, Daubin,
  and Boussau]{STD14}
  G.~J. Sz{\"o}ll{\H{o}}si, E.~Tannier, V.~Daubin, and B.~Boussau.
  \newblock The inference of gene trees with species trees.
  \newblock \emph{Systematic biology}, 64\penalty0 (1):\penalty0 e42--e62, 2014.
  
  \bibitem[Tellier and Lemaire(2014)]{LT14}
  A.~Tellier and C.~Lemaire.
  \newblock Coalescence 2.0: {A} multiple branching of recent theoretical
  developments and their applications.
  \newblock \emph{Molecular ecology}, 23\penalty0 (11):\penalty0 2637--2652,
  2014.
  
  \bibitem[Volz et~al.(2013)Volz, Koelle, and Bedford]{VKB13}
  E.~M. Volz, K.~Koelle, and T.~Bedford.
  \newblock Viral phylodynamics.
  \newblock \emph{PLoS Computational Biology}, 9\penalty0 (3):\penalty0 e1002947,
  2013.
  
\end{thebibliography}

\end{document}